\documentclass[a4paper,12pt]{amsart}
\usepackage{amssymb}
\usepackage{graphicx}
\usepackage{amsmath}
\usepackage{booktabs}
\usepackage{bbm}
\usepackage{color}

\newtheorem{teorema}{Theorem}
\newtheorem{lema}{Lemma}

\newtheorem{prop}{Proposition}

\newcommand{\ZZ}{\mathbb{Z}}
\newcommand{\RR}{\mathbb{R}}
\newcommand{\CC}{\mathbb{C}}

\newcommand{\Ss}{\mathbb{S}}
\newcommand{\TT}{\mathbb{T}}
\newcommand{\OO}{\mathbb{O}}
\newcommand{\one}{\mathbbm{1}}

\newcommand{\rr}{{\mathcal R}}
\newcommand{\ri}{{\mathcal I}}

\newcommand{\Real}{{\mathcal R}e}
\newcommand{\Imag}{{\mathcal I}m}

\newcommand{\tetra}{\langle\TT ,\kappa\rangle}
\newcommand{\cubo}{\langle\OO ,-Id\rangle}
\DeclareMathOperator{\fix}{Fix}

\begin{document}

\title[Hopf bifurcation with tetrahedral and octahedral symmetry]{Hopf bifurcation with tetrahedral and octahedral symmetry}

\author[I.S. Labouriau, A.C. Murza]{Isabel Salgado Labouriau, Adrian C. Murza}
\address{I.S. Labouriau --- Centro de Matem\'atica da Universidade do Porto.\\ Rua do Campo Alegre 687, 4169-007 Porto, Portugal\\ A.C. Murza --- Departamento de Matem\'atica, Faculdade de Ci\^encias\\
Universidade do Porto.\\ Rua do Campo Alegre 687, 4169-007 Porto, Portugal}

\date{\today}
\begin{abstract}
In the study of the periodic solutions of a $\Gamma$-equivariant dynamical system, the $H~\mathrm{mod}~K$ theorem gives all possible periodic solutions, based on  group-theoretical aspects. By contrast, the equivariant Hopf theorem guarantees the existence of families of small-amplitude periodic solutions bifurcating from the origin for each $\mathbf{C}$-axial subgroup of $\Gamma\times\Ss^1$. 
In this article we compare the bifurcation of periodic solutions for generic differential equations equivariant under   the full group of symmetries of the tetrahedron and the group of rotational symmetries of the cube.
The two groups are the image of inequivalent representations of the symmetric group $S_4$.
The possible  spatial symmetries of bifurcating solutions are different, even though the two groups yield the same group of matrices  $\Gamma\times\Ss^1$.
The same group of matrices occurs again as the extension $\Gamma\times\Ss^1$ when  $\Gamma$ is  the full group of symmetries of the cube.
 For these three groups,  while characterizing the Hopf bifurcation, we identify which periodic solution types, whose existence is guaranteed by the $H~\mathrm{mod}~K$ theorem, are obtainable by Hopf bifurcation from the origin.
\end{abstract}

\maketitle



\section{Introduction}\label{section Introduction}
The formalism of $\Gamma$-equivariant differential equations, ie. those equations whose associated vector field commutes with the action of a  compact
 group $\Gamma$ has been developed by Golubitsky, 
Stewart and Schaeffer in \cite{GS85}, \cite{GS88} and \cite{GS03}. Within this formalism, two methods for obtaining periodic solutions have been described: the $H~\mathrm{mod}~K$ theorem   \cite[Ch.3]{Buono,GS03} and the equivariant Hopf theorem  \cite[Ch.4]{GS03}. 
The equivariant Hopf theorem guarantees the existence of families of small-amplitude periodic solutions bifurcating from the origin for all $\mathbf{C}$-axial subgroups of $\Gamma\times\Ss^1$, under some generic conditions.
The $H~\mathrm{mod}~K$ theorem offers the complete set of possible periodic solutions based exclusively on the structure of the group $\Gamma$ acting on the differential equation.
 It also guarantees the existence of a model with this symmetry having these periodic solutions, but it is not an existence result for any specific equation. 
 
 In this article, we discuss  the cases when $\Gamma$ is the tetrahedral group $\tetra$ of symmetries of the  tetrahedron and  when $\Gamma$ is the octahedral group $\OO$ of  rotational symmetries of the cube.
As abstract groups, $\tetra$ and $\OO$ are isomorphic to the symmetric group $S_4$,
but they correspond to nonequivalent representations both in $\RR^3$ and $\CC^3$,
generating nonconjugate subgroups of $\mathbf{O}(3)$ and $\mathbf{U}(3)$, respectively.
Identifying $\CC^3$ to $\RR^6$ yields two nonconjugate subgroups of $\mathbf{O}(6)$.

For Hopf bifurcation, the relevant action is that of $\Gamma\times\Ss^1$, where $\theta\in\Ss^1$ acts by multiplication by $e^{i\theta}$ in $\CC^3\sim\RR^6$.
It turns out that the matrix groups $\tetra\times\Ss^1$ and $\OO\times\Ss^1$ coincide, although they correspond to inequivalent representations of $S_4\times\Ss^1$.
The analysis of equivariant Hopf bifurcation is the same for both groups, but the  interpretation is not the same for the two different representations of $S_4$ and hence the periodic orbits have different spatial and spatio-temporal symmetries.
In addition, for  the full  group $\Gamma=\cubo$ of symmetries of the cube, the matrix group $\Gamma\times\Ss^1$ is also the same as those above. 
As a consequence, we have one setup for bifurcation in $\RR^6$ with three different interpretations depending on which  representation we consider.

This raises the general question of the dynamical interpretation of equivariant bifurcation data for different representations of the same group that correspond to the same matrix group.
A simple example is the group $\ZZ_3$ where the generator $\gamma\in\ZZ_3$ is mapped into $e^{2\pi i/3}$ in one representation, and to $e^{4\pi i/3}$ in another. The representations are not equivalent, but their image is the same matrix group and {\sl a fortiori} the extension  to $\ZZ_3\times\Ss^1$ has the same property.
The situation in this article is even more special, in that the representations of $S_4$ have different images, but the image matrix groups for the extensions $S_4\times\Ss^1$ coincide, and they also coincide with the image of the extension of another group.
A similar situation  arises if the matrix groups of two inequivalent representations do not coincide but are conjugate within the group $\mathbf{O}(n)$. 
The  conclusion would be the same, but it would not be so evident because the matrix groups would be different.
The results in the present article  indicate that this is a general problem worth studying.

\bigbreak

Steady-state bifurcation problems with octahedral symmetry  are analysed by Melbourne \cite{Melbourne} using results from singularity theory. 
For non-degenerate bifurcation problems equivariant with respect to the standard action on $\RR^3$ of the octahedral group he finds three branches of symmetry-breaking steady-state bifurcations corresponding to the three maximal isotropy subgroups with one-dimensional fixed-point subspaces. 
Hopf bifurcation with the rotational symmetry of the tetrahedron is studied by Swift and Barany \cite{barany},
motivated by problems in fluid dynamics.
 They find periodic branches created at Hopf bifurcation, as well as 
 evidence of chaotic dynamics, bifurcating directly from the equilibrium --- instant chaos.
Generic Hopf bifurcation with the
rotational symmetries of the cube is studied by Ashwin and Podvigina \cite{Ashwin}, also with the motivation of fluid dynamics. 
They also find evidence of chaotic dynamics, this time arising from secondary bifurcations from periodic branches created at Hopf bifurcation.
Also relevant for this study are the extensive discussions of subgroups of $\mathbf{O}(3)$ and $\mathbf{SO}(3)$  in articles by Ihrig and Golubitsky \cite{IG} and by Chossat and Lauterbach \cite{CLO3}, as well as in the books by Golubitsky, 
Stewart and Schaeffer  \cite{GS88} and by Chossat and Lauterbach \cite{CL}.

Solutions predicted by the $H~\mathrm{mod}~K$ theorem cannot always be obtained by a generic Hopf bifurcation from the trivial equilibrium. When the group is finite abelian, the periodic solutions whose existence is allowed by the $H~\mathrm{mod}~K$ theorem that are realizable from the equivariant Hopf theorem are described in \cite{Abelian Hopf}.

In this article, we pose a more specific question: which periodic solutions predicted by the $H~\mathrm{mod}~K$ theorem are obtainable by  Hopf bifurcation from the trivial steady-state when  $\Gamma$ is one of the groups discussed above?

We will answer this question by finding for each group that not all periodic solutions predicted by the $H~\mathrm{mod}~K$ theorem occur at  Hopf bifurcations from the trivial equilibrium. 
For this we analyse  bifurcations taking place in four-dimensional invariant subspaces and giving rise to periodic solutions with very small symmetry groups.
In particular, we find that 
some solutions predicted by the  $H~\mathrm{mod}~K$ theorem for $\tetra$-equivariant vector fields 
are not compatible with the $\Gamma\times\Ss^1$ action, and if they bifurcate from an equilibrium, they either have more symmetry or they arise at a resonant Hopf bifurcation.
This is also the case for $\cubo$.

\subsection*{Framework of the article}
The relevant actions of the groups $\Gamma=\tetra$ of symmetries of the tetrahedron, $\Gamma=\OO$ of rotational symmetries of the cube, $\Gamma=\cubo$ of all symmetries of the cube and of $\Gamma\times\Ss^1$ on $\RR^6\sim\CC^3$ are described in Section~\ref{section 1}, after stating some preliminary results and definitions in Section~\ref{secPreliminary}.
Hopf bifurcation is treated in Section~\ref{secHopf}, where we  summarise some of the results  of Swift and Barany \cite{barany}  on $\TT\times\Ss^1$ and of
Ashwin and Podvigina \cite{Ashwin} on $\OO\times\Ss^1$, together with the formulation of the same results 
 when the group is interpreted as $\tetra\times\Ss^1$ and as $\cubo\times\Ss^1$.
 This includes the analysis of Hopf bifurcation inside  fixed-point subspaces for submaximal isotropy subgroups, one of which we perform in more detail than in \cite{barany,Ashwin}, giving a geometric proof of the existence of  up to
 three 
branches of   submaximal periodic solutions, for some values of the parameters in a degree three normal form.
Finally, we apply the
$H~\mathrm{mod}~K$ theorem in Section~\ref{sectionSpatioTemporal} where we compare the bifurcations for the  three group actions.

\section{Preliminary results and definitions}\label{secPreliminary}
 We start by giving some definitions from  Golubitsky and Stewart's book \cite{GS03}. 
The reader is referred to this book and to Chossat and Lauterbach's  \cite{CL} for results on bifurcation with symmetry.

Let $\Gamma$ be a compact Lie group. A representation  of  $\Gamma$ on a vector space $W$ is \emph{$\Gamma$-simple } if either:
\begin{itemize}
\item [(a)] $W\sim V\oplus V$ where $V$ is absolutely irreducible for $\Gamma,$ or
\item [(b)]  the action of $\Gamma$ on $W$ is irreducible but not absolutely irreducible.
\end{itemize}

Let $W$ be a $\Gamma$-simple representation and let
$f$ be a $\Gamma$-equivariant vector field in $W$.
Then it follows \cite[Ch. XVI, Lemma1.5]{GS88} that
% if $f$ is a $\Gamma$-equivariant vector field, and
  if the Jacobian matrix $(df)_{0}$ of $f$ evaluated at the origin has purely imaginary eigenvalues $\pm \omega i$,
 then in suitable coordinates $(df)_{0}$ has the form:
 $$
(df)_{0}= \omega J=\omega \begin{bmatrix}
0&-Id\\
Id&0
\end{bmatrix}
$$
where $Id$ is the identity matrix.
Consider the action of $\Ss^1$  on $W$ given by $\theta x=e^{i\theta J}x$.
A subgroup $\Sigma\subseteq\Gamma\times\Ss^1$ is \emph{$\CC$-axial} if $\Sigma$ is an isotropy subgroup and
$\dim\fix(\Sigma)=2$.

Let $\dot{x}=f(x)$ be a $\Gamma$-equivariant differential equation with a $T$-periodic solution $x(t).$ We call $(\gamma,~\theta)\in\Gamma\times\Ss^1$ a spatio-temporal symmetry of the solution $x(t)$ if $\gamma\cdot x(t+\theta)=x(t)$  for all $t$. A spatio-temporal symmetry of the solution $x(t)$ for which $\theta=0$ is called a spatial symmetry, since it fixes the point $x(t)$ at every moment of time.
 
 The main tool here will be the following theorem.

\begin{teorema}
[Equivariant Hopf Theorem \cite{GS03}] Let a compact Lie group $\Gamma$ act $\Gamma$-simply, orthogonally and nontrivially on  $\RR^{2m}$.
Assume that
\begin{itemize}
\item [(a)] $f:\RR^{2m}\times\RR\rightarrow\RR^{2m}$ is $\Gamma$-equivariant. Then $f(0,\lambda)=0$ and $(df)_{0,\lambda}$ has eigenvalues $\sigma(\lambda)\pm i\rho(\lambda)$ each of multiplicity $m;$
\item [(b)] $\sigma(0)=0$ and $\rho(0)=1;$
\item [(c)] $\sigma'(0)\neq0$ the eigenvalue crossing condition;
\item [(d)] $\Sigma\subseteq\Gamma\times\Ss^1$ is a $\CC$-axial subgroup.
\end{itemize}
Then there exists a unique branch of periodic solutions with period $\approx2\pi$ emanating from the origin, with spatio-temporal symmetries $\Sigma.$
\end{teorema}

 The group of all spatio-temporal symmetries of $x(t)$ is denoted
$\Sigma_{x(t)}\subseteq\Gamma\times\Ss^1$.
The symmetry group $\Sigma_{x(t)}$ can be identified with a pair of subgroups $H$ and $K$ of $\Gamma$ and a homomorphism $\Phi:H\rightarrow\Ss^1$ with kernel $K.$ We define
$$
H=\left\{\gamma\in\Gamma:\gamma \{x(t)\}=\{x(t)\}\right\}
\qquad
K=\left\{\gamma\in\Gamma:\gamma x(t)=x(t)~\forall t\right\}
$$
where $K\subseteq\Sigma_{x(t)}$ is the subgroup of spatial symmetries of $x(t)$ and the subgroup $H$ of $\Gamma$ consists of  symmetries preserving the trajectory  $x(t)$ but not necessarily the points in the trajectory. 
We abuse notation saying that $H$ is a group of spatio-temporal symmetries of $x(t)$. 
This makes sense because the groups $H\subseteq\Gamma$ and $\Sigma_{x(t)}\subseteq\Gamma\times\Ss^1$ are isomorphic; the isomorphism being the restriction to $\Sigma_{x(t)}$ of the projection of $\Gamma\times\Ss^1$ onto $\Gamma$.

Given an  isotropy subgroup $\Sigma\subset\Gamma$, denote by
$N(\Sigma)$  the normaliser of $\Sigma$ in $\Gamma$, satisfying $N(\Sigma)=\{\gamma\in\Gamma: \gamma \Sigma=\Sigma\gamma\}$, 
and by $L_\Sigma$  the  variety $L_\Sigma=\bigcup_{\gamma\notin \Sigma}\fix(\gamma)\cap\fix(\Sigma)$.

The second important tool in this article is the following result.

\begin{teorema}\label{teorema H mod K}
($H~\mathrm{mod}~K$ Theorem \cite{Buono,GS03}) Let $\Gamma$ be a finite group acting on $\RR^n.$ There is a periodic solution to some $\Gamma$-equivariant system of ODEs on $\RR^n$ with spatial symmetries $K$ and spatio-temporal symmetries $H$ if and only if the following conditions hold:
\begin{itemize}
\item [(a)] $H/K$ is cyclic;
\item [(b)] $K$ is an isotropy subgroup;
\item [(c)] $dim~Fix(K)\geqslant2.$ If $dim~Fix(K)=2,$ then either $H=K$ or $H=N(K)$;
\item [(d)] $H$ fixes a connected component of $\mathrm{Fix(K)\backslash L_K}$.
\end{itemize}
Moreover, if $(a)-(d)$ hold, the system can be chosen so that the periodic solution is stable.
\end{teorema}

 When $H/K\sim\ZZ_m$, the periodic solution $x(t)$ is called either a standing wave or (usually for $m\geqslant3$) a discrete rotating wave; and when $H/K\sim \Ss^1$ it is called a rotating wave  \cite[page 64]{GS03}. Here all rotating waves are discrete.

\section{Group actions}\label{section 1}
Our aim in this article is to compare the bifurcation of periodic solutions for generic differential equations equivariant under   different representations of the same  abstract group. 
In this section we describe the  representations used.

\subsection{Symmetries of the  tetrahedron}
The group $\TT $ of rotational symmetries of the tetrahedron \cite{barany}  has
order $12$. Its action on $\RR^3$ is generated by two rotations $R$ and $C$ of orders 2 and 3, respectively, and given by
\begin{equation*}\label{generators1}
R=
\begin{bmatrix}
1&0&0\\
0&-1&0\\
0&0&-1
\end{bmatrix}
\hspace{0.3cm}C=
\begin{bmatrix}
0&1&0\\
0&0&1\\
1&0&0
\end{bmatrix}.
\end{equation*}

Next we want to augment the group $\TT $ with a reflection, given by
\begin{equation*}\label{generator}
\kappa=
\begin{bmatrix}
1&0&0\\
0&0&1\\
0&1&0
\end{bmatrix}
\end{equation*}
to form $\tetra$, the full group of symmetries of the thetrahedron,
that  has order $24$.
The matrix group $\tetra$ corresponds to the irreducible complex representation $\rho_1$ of $S_4$ given by permutations of the indices of the  vertices $V_1=(1,1,1)$, $V_2=(1,-1,-1)$, $V_3=(-1,-1,1)$, $V_4=(-1,1,-1)$ of a tetrahedron,
and is given by: 
$$
\rho_1\left((234)\right)= C\qquad\qquad
\rho_1\left((12)(34)\right)= R\qquad\qquad
\rho_1\left((34)\right)= \kappa \ .$$

We obtain an action of $\tetra$ on $\RR^6$ by identifying $\RR^6\equiv\CC ^3$ and taking the same matrices as generators,
with 1 and 0 standing for  $2\times 2$ identity and zero blocks, respectively. Decomposing $\RR^6\equiv\CC ^3$ into the subspaces of real and imaginary parts shows that the representation of $\tetra$ on $\RR^3$ is
$\tetra$-simple.
The isotropy lattice of $\tetra$ is shown in Figure~\ref{FigLatticeTetra2}.

\begin{figure}[ht]
\begin{center}
\includegraphics[scale=0.7]{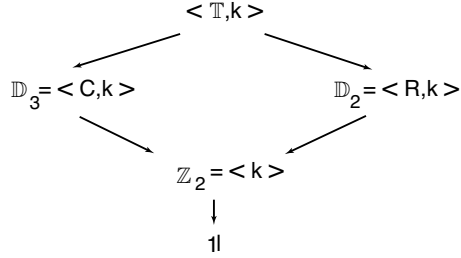} 
\end{center}
\caption{Isotropy lattice for the group  $\tetra $ of symmetries of the tetrahedron.}\label{FigLatticeTetra2}
\end{figure}

\subsection{Symmetries of the cube}
The  action on $\RR^3$  of the group $\OO$ of rotational symmetries of the cube is generated by the rotation  $C$ of order 3,  with the matrix above, and by the rotation $T$ of order 4
$$
T=
\begin{bmatrix}
0&1&0\\
-1&0&0\\
0&0&1
\end{bmatrix} .
$$
As in the case of the symmetries of the tetrahedron, we obtain an action of $\OO$ on $\RR^6$ by identifying $\RR^6\equiv\CC ^3$ and using the same matrices,  with the conventions above.

 Geometrically, $\OO$ can be seen as  permutations of the indices $j=1,\ldots,4$ of opposite pairs $\pm P_j$ of the vertices  $P_1=(1,1,1)$, $P_2=TP_1$, $P_4=TP_2$, $P_3=TP_4$ and $P_{4+j}=-P_j$, $j=1,\ldots,4$ of a cube. This corresponds to 
 a complex representation $\rho_2$ of $S_4$ given by
$$
\rho_2\left((234)\right)=C\qquad\qquad
\rho_2\left((1243)\right)=T \ .
$$
This means that as abstract groups,  $\tetra$ and $\OO$ are isomorphic, the isomorphism maps $C$ into itself and the rotation $T$ of order 4 in $\OO$ into the rotation-reflection $C^2R\kappa$ in $\tetra$. 
However,
 the two representations of $S_4$  are not equivalent, since $\OO$ is a subgroup of   $SO(6)$
  and $\tetra$ is not. The isotropy lattice of  $\OO$ is shown in Figure~\ref{FigLatticeOcta}.

\begin{figure}[ht]
\begin{center}
\includegraphics[scale=0.65]{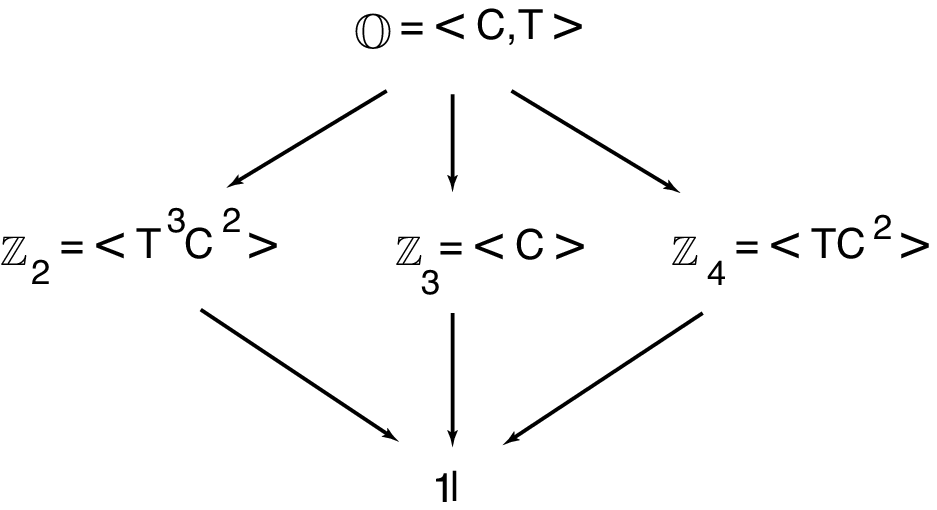}\ \includegraphics[scale=0.65]{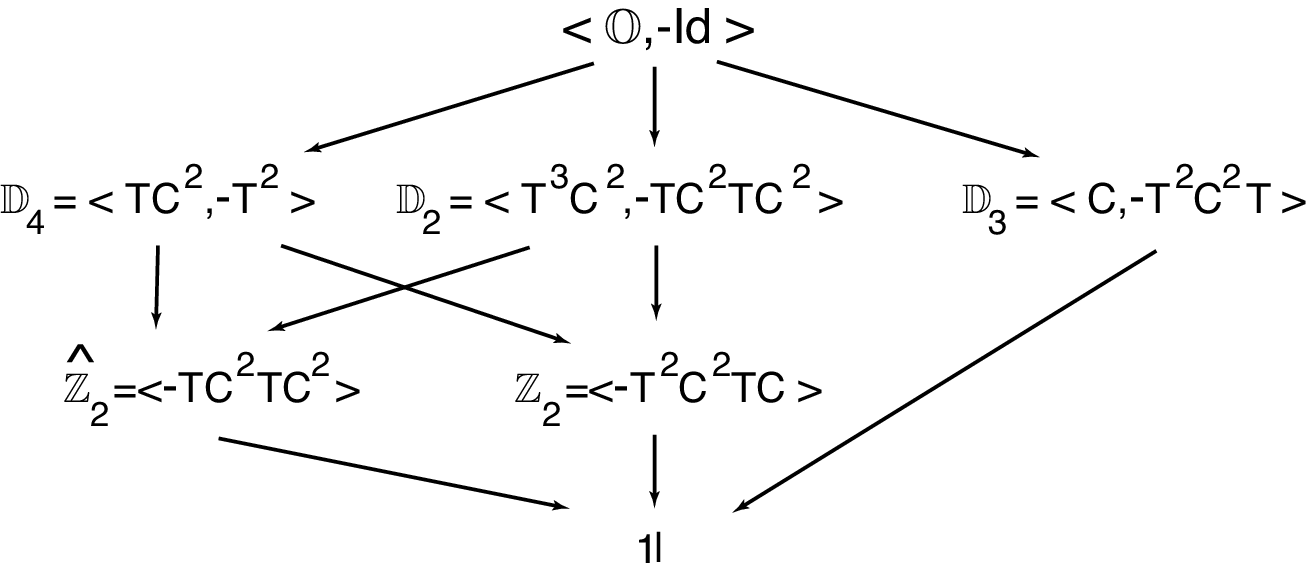} \end{center}
\caption{Isotropy lattices for the group  $\OO$ of rotational symmetries of the cube (left) and for the full group of symmetries of the cube $\cubo$ (right).}\label{FigLatticeOcta}
\end{figure}

 Adding the rotation-reflection $-Id$ to $\OO$ yields the full symmetry group of the cube, that will be denoted $\cubo$, the image of the extension of $\rho_2$ to a representation $\rho_3$ of $S_4\times\ZZ_2$. The group $\cubo$
 has order 48 and its isotropy lattice is shown in  Figure~\ref{FigLatticeOcta}.

\subsection{Adding $\Ss^1$}
The corresponding actions of $\Gamma\times\Ss^1$ on $\CC ^3$, where $\Gamma$ is one of the groups $\tetra$, $\OO$ and $\cubo$, are obtained by adding the elements $e^{i\theta}\cdot\mathrm{Id},~\theta\in(0,2\pi)$
to the group $\Gamma$.
Note that with this action the elements of $\Ss^1$ commute with those of $\Gamma$, hence this extends the representations $\rho_j$ above to  representations $\tilde{\rho}_1$ and $\tilde{\rho}_2$ of $S_4\times \Ss^1$ and $\tilde{\rho}_3$ of $S_4\times\ZZ_2\times \Ss^1$.

 Since $e^{i\pi}Id=-Id\in\cubo$, 
% then 
 the groups $\cubo\times\Ss^1$ and $\OO\times\Ss^1$ coincide as subgroups of ${\mathbf {O}}(6)$. 
The product  $\tetra\times\Ss^1$ also yields the same subgroup of ${\mathbf {O}}(6)$,
 because $T=e^{i\pi}C^2R\kappa\in \tetra\times\Ss^1$ and, on the other hand,
 $R= TC^2TC^2\in  \OO$ and  $\kappa=e^{i\pi} T^2C^2TC\in \OO\times\Ss^1$.
We will denote this  matrix  group by $\Gamma\times\Ss^1$, its  isotropy lattice  is shown in 
Figure~\ref{FigLatticeGammaS1}.

Since the two representations $\rho_1$ and $\rho_2$ are not equivalent, their characters $\chi_1$ and $\chi_2$
are orthogonal maps in the space $L^2\left(S_4,\CC\right)$.
The characters $\tilde{\chi}_j$ of the extended representations $\tilde{\rho}_j$ are given by 
$\tilde{\chi}_j\left((\gamma,\theta)\right)=e^{i\theta}\chi_j(\gamma)$, and hence, $\tilde{\chi}_1$ and  $\tilde{\chi}_2$ are also orthogonal  in the space $L^2\left(S_4\times\Ss^1,\CC\right)$.
It follows that the two representations $\tilde{\rho}_1$ and $\tilde{\rho}_2$ are not equivalent, even though their images in $O(6)$ coincide as matrix groups.

Note that in the case $\Gamma=\cubo$  the element $(-Id,\pi)$  acts trivialy as the identitiy.
The representation $\tilde{\rho}_3$ is not faithful, its kernel  is generated by $(Id,-Id,\pi)\in S_4\times\ZZ_2\times \Ss^1$.

\begin{figure}[ht]
\begin{center}
 \includegraphics[scale=0.8]{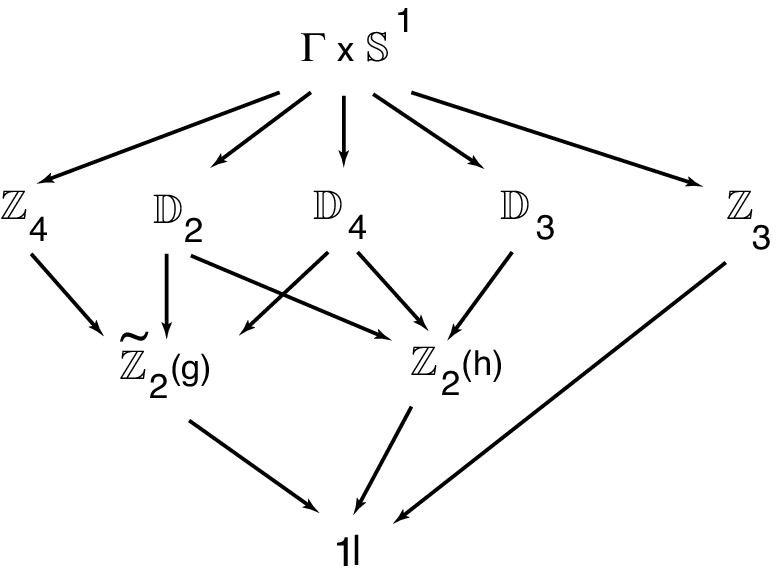}
\end{center}
\caption{Isotropy lattice for the group  $\Gamma\times\Ss^1$, where $\Gamma$ is any of the three groups $\tetra$, $\OO$ or $\cubo$. The labels (g) and (h) on the two subgroups isomorphic to $\ZZ_2$ correspond to rows in Tables~\ref{table isotropy subgroups tetrahedral+k} and \ref{table isotropy subgroups} below.}\label{FigLatticeGammaS1}
\end{figure}

\section{Hopf bifurcation}\label{secHopf}
The first step in studying $\Gamma$-equivariant Hopf bifurcation is to obtain the $\CC$-axial subgroups of $\Gamma\times\Ss^1$.
 The fixed-point subspaces for the isotropy subgroups are listed in Table~\ref{table isotropy subgroups tetrahedral+k}. Generators of the isotropy subgroups are given in Table~\ref{table isotropy subgroups}, in terms of the generators of $\tetra$, of $\OO$ and of $\cubo$.
\begin{table}[h]
\caption{Solution types, isotropy subgroups and  fixed-point subspaces for the action of $\Gamma\times\Ss^1$  on $\CC^3$ where $\Gamma$ is either $\tetra$ or $\OO$ or $\cubo$. Here $\omega$ stands for $e^{2\pi i/3}$.
The dividing line separates $\CC$-axial subgroups.
}
\label{table isotropy subgroups tetrahedral+k}
\begin{tabular}{|c|l|c|c|c|}
\hline
Index&Name&Isotropy&Fixed-point &dim\\
&&subgroup&subspace&\\
\hline
$(a)$&Origin &$\Gamma\times\Ss^1$& $\{(0,0,0)\}$&$0$\\
$(b)$&Pure mode &$\mathbb{D}_4$&$\{(z,0,0)\}$&$2$\\
$(c)$& Standing wave& $\mathbb{D}_3$&$\{(z,z,z)\}$&$2$\\
$(d)$& Rotating wave& $\ZZ_3$ &$\{(z,\omega z,\omega^2z)\}$&$2$\\
$(e)$& Standing wave & $\mathbb{D}_2$&$\{(0,z,z)\}$&$2$\\
$(f)$ &Rotating wave&$  \ZZ_4$&$\{(z,iz,0)\}$&$2$\\
\hline
$(g)$ &2-Sphere solutions&$\widetilde{\ZZ}_2$& $\{(0,z_1,z_2)\}$&$4$\\
$(h)$ &2-Sphere~solutions&$\ZZ_2$&$\{(z_1,z_2,z_2)\}$&$4$\\
$(i)$ &General~solutions&$\one$&$\{(z_1,z_2,z_3)\}$&$6$\\
\hline
\end{tabular}
\end{table}

\begin{table}[h]
\caption{
Generators of isotropy subgroups for the action of $\Gamma\times\Ss^1$  on $\CC^3$ for $\Gamma=\tetra$, $\Gamma=\OO$ and $\Gamma=\cubo$, with $\omega=e^{2\pi i/3}$.
The index (a) to (i) refers to the rows in Table~\ref{table isotropy subgroups tetrahedral+k}, and the dividing line separates $\CC$-axial subgroups.
}
\label{table isotropy subgroups}
\begin{tabular}{|c|c|c|c|c|}
\hline
Index&Isotropy&Generators&Generators&Generators\\
&subgroup& in $\langle\TT,\kappa\rangle\times\Ss^1$&in $ \mathbb{O}\times\mathbb{S}^1$&in $\cubo\times\mathbb{S}^1$\\
\hline
$(a)$  &$\Gamma\times\Ss^1$
& $\{C,R,\kappa,e^{i\theta}\}$&$\{C,T,e^{i\theta}\}$&$\{C,T,e^{i\theta}\}$ \\
$(b)$  &$\mathbb{D}_4$&$\{e^{\pi i}  C^2RC,\kappa\}$ &$\{TC^2,e^{\pi i}T^2\}$&$\{TC^2,-T^2\}$ \\
$(c)$ &$\mathbb{D}_3$&$\{C,\kappa\}$&$\{C,e^{\pi i}T^2C^2T\}$ &$\{C,-T^2C^2T\}$  \\
$(d)$ &$\ZZ_3$&$\{\bar{\omega} C\}$&$\{\bar{\omega}C\}$ &$\{\bar{\omega}C\}$\\
$(e)$&$\mathbb{D}_2$&$\left\{e^{\pi i}  R,\kappa \right\}$&$\{e^{\pi i}TC^2TC^2,T^3C^2\}$&$\{-TC^2TC^2,T^3C^2\}$\\
$(f)$ &$  \ZZ_4$&$ \left\{    e^{\pi i/2}  C^2\mathrm{R}\kappa\right\}$&$\{e^{-\pi i/2}T\}$&$\{e^{-\pi i/2}T\}$\\
\hline
$(g)$ &$\widetilde{\ZZ}_2$&$\{e^{\pi i}  R\}$&$\{e^{\pi i}TC^2TC^2\}$&$\{-TC^2TC^2\}$\\
$(h)$ &$\ZZ_2$&$\{\kappa\}$&$\{e^{\pi i}T^2C^2TC\}$&$\{-T^2C^2TC\}$\\
$(i)$ &$\one$&$\{Id\}$&$\{Id\}$&$\{Id\}$\\
\hline
\end{tabular}
\end{table}

The normal form for a  $\Gamma\times\Ss^1$-equivariant vector field truncated to the cubic order is
\begin{equation}\label{normal form}
\left\{
\begin{array}{l}
\dot{z}_1=z_1\left(\lambda+\gamma\left(|z_1|^2+|z_2|^2+|z_3|^2\right)+\alpha\left(|z_2|^2+|z_3|^2\right)\right)
+\bar{z}_1\beta\left(z_2^2+z_3^2\right)\\
\\
\dot{z}_2=z_2\left(\lambda+\gamma\left(|z_1|^2+|z_2|^2+|z_3|^2\right)+\alpha\left(|z_1|^2+|z_3|^2\right)\right)
+\bar{z}_2\beta\left( z_1^2+ z_3^2\right)\\
\\
\dot{z}_3=z_3\left(\lambda+\gamma\left(|z_1|^2+|z_2|^2+|z_3|^2\right)+\alpha\left(|z_1|^2+|z_2|^2\right)\right)
+\bar{z}_3\beta \left(z_1^2+z_2^2\right)
\end{array}
\right.
\end{equation}
where $\alpha,\beta,\gamma,\lambda$ are all complex coefficients. 
This normal form is the same used in \cite{barany} for the $\TT\times\Ss^1$ action, except that  the extra symmetry $\kappa$ forces some of the coefficients to be equal. The normal form \eqref{normal form} is slightly different, but equivalent, to the one given in \cite{Ashwin}
 for the group $\OO\times\Ss^1$.
As in \cite{Ashwin,barany}, the origin is always an equilibrium of  \eqref{normal form} and it undergoes a Hopf bifurcation when $\lambda$ crosses the imaginary axis.
By the Equivariant Hopf Theorem, this generates several branches of periodic solutions, corresponding to the   $\CC$-axial subgroups of  $\Gamma\times\Ss^1$.
 
Under additional conditions on the parameters in \eqref{normal form} there may  be other periodic solution branches arising through Hopf bifurcation outside the fixed-point subspaces for $\CC$-axial subgroups.
These have been analysed in \cite{Ashwin}, we proceed to describe them briefly, with some additional information from \cite{MurzaThesis}.

\subsection{Submaximal branches in $\{(z_1,z_2,0)\}$}\label{FixPiR}
As a fixed-point subspace for the $\tetra\times\Ss^1$ action, this subspace  is 
 $\fix\left(\ZZ_2\left( e^{\pi i}CRC^2\right)\right)$  conjugate to 
$\fix\left(\ZZ_2\left( e^{\pi i}R\right)\right)$,  (see Tables~\ref{table isotropy subgroups tetrahedral+k} and \ref{table isotropy subgroups}). 
The same subspace  is described as $\fix\left(\ZZ_2\left(e^{\pi i}T^2 \right)\right)$, under either the 
$\OO\times\Ss^1$ or the $\cubo\times\Ss^1$ actions,
and is  conjugate to  the subspaces $\fix\left(\ZZ_2\left(e^{\pi i}TC^2TC^2 \right)\right)$,
 and $\fix\left(\ZZ_2\left(-TC^2TC^2 \right)\right)$, that appear in Table~\ref{table isotropy subgroups}. 
It contains the fixed-point subspaces $\{(z,0,0)\}$, $\{(z,z,0)\}$, $\{(iz,z,0)\}$, corresponding to $\CC$-axial subgroups  of $\Gamma\times\Ss^1$, as well as a conjugate copy of each one of them.

Solution branches of \eqref{normal form}  in the fixed-point subspace $\{(z_1,z_2,0)\}$
 have been analysed in  \cite{MurzaThesis,Ashwin,Swift}.
   The analysis in \cite{Swift} is quite extensive, and is much more complete than the analysis in either \cite{Ashwin} or the present article.
 We give an outline of their findings here, so results can be used in Section~\ref{sectionSpatioTemporal}. 
 
Restricting the  normal form \eqref{normal form} to  the subspace $\{(z_1,z_2,0)\}$ and eliminating solutions that lie in the two-dimensional fixed-point subspaces, one finds that   for $\beta\ne 0$ there is a solution branch with no additional symmetries, if and only if both $\displaystyle\left|\alpha/\beta\right|> 1$ and $\displaystyle-1<\Real\left(\alpha/\beta\right)<1$ hold.
These solutions have the form $\{(\xi z,z,0)\}$, with $\xi=re^{i\phi}$, where
\begin{equation}\label{pitchfork condition}
 \cos(2\phi)=-\Real\left(\alpha/\beta\right)\quad
 \sin(2\phi)=\pm\sqrt{1-\left( \Real\left(\alpha/\beta\right) \right)^2}
 \quad
 r^2=\frac{\displaystyle \Imag\left(\alpha/\beta\right)+\sin(2\phi)}{\displaystyle \Imag\left(\alpha/\beta\right)-\sin(2\phi)}.
 \end{equation}
 
 The submaximal branch of periodic solutions connects all the maximal branches that lie in the subspace 
 $\{(z_1,z_2,0)\}$. This can be deduced from the expressions \eqref{pitchfork condition}, as we proceed to explain, and as illustrated in Figure~\ref{figuraPitchfork}.
 
\begin{figure}[ht]
\begin{center}
\includegraphics[scale=0.350]{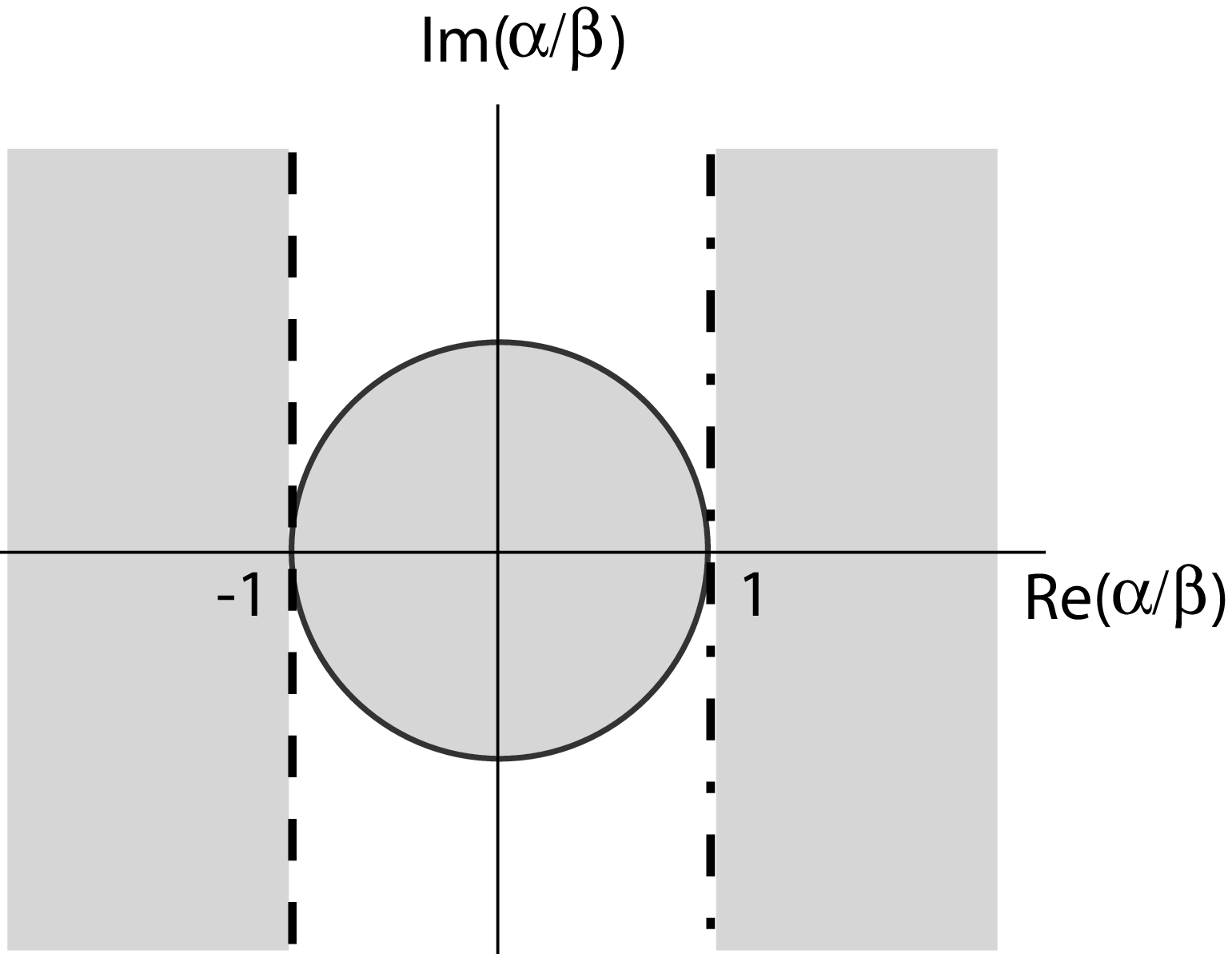}\quad
\includegraphics[scale=0.350]{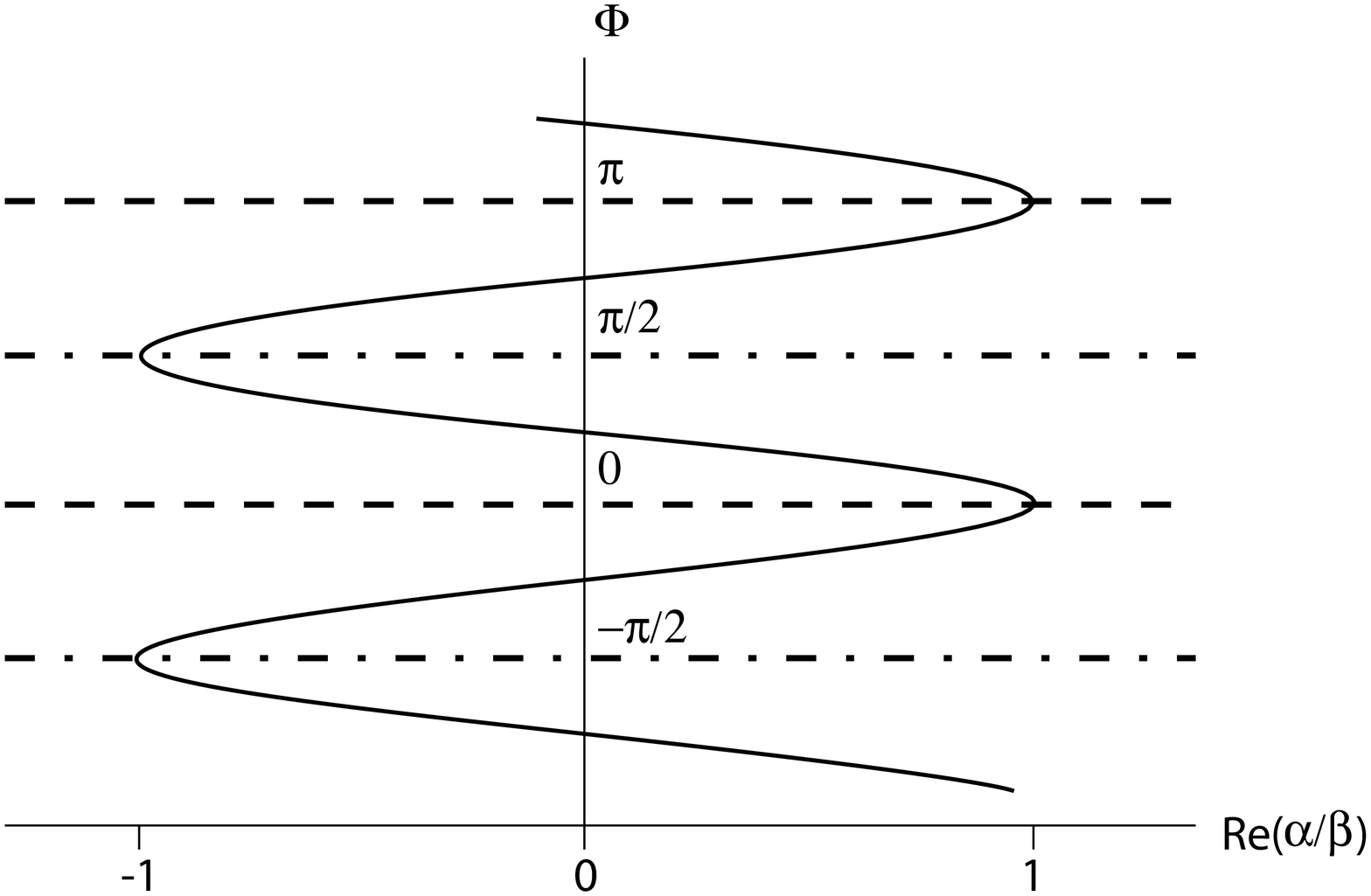}
\caption{Submaximal solutions in the subspace $\{(z_1,z_2,0)\}$ and their connection to the maximal branches. 
Left: submaximal branches exist in the white area of the $\alpha/\beta$-plane. 
Bifurcation into the branches that lie in $\{(0,z,0)\}$ and $\{(z,0,0)\}$, 
occurs at the circle $\left| \alpha/\beta\right|=1$, for  any value of   the phase-shift $\phi$.
Right: 
dashed lines stand for the branches in $\{(\pm z,z,0)\}$ and dot-dash for the branches in $\{(\pm i z,z,0)\}$. Pairs of submaximal solutions bifurcate from  these branches at pitchforks at specific values of  the phase-shift $\phi$ when $\Real\left(\alpha/\beta\right)=\pm 1$.}\label{figuraPitchfork}
\end{center}
\end{figure} 
When $\Real\left(\alpha/\beta\right)=+1$ we get $\xi=\pm i$ and the submaximal  branches lie in the subspaces  $\{(\pm iz,z,0)\}$. 
The submaximal branches only exist for $\Real\left(\alpha/\beta\right)\le 1$
and, when $\Real\left(\alpha/\beta\right)$ increases to $+1$, pairs of submaximal branches coalesce into the subspaces $\{(\pm iz,z,0)\}$, as in Figure~\ref{figuraPitchfork}.

Similarly, for $\Real\left(\alpha/\beta\right)=-1$ we have $\xi=\pm 1$. 
As $\Real\left(\alpha/\beta\right)$ decreases to $-1$, the submaximal branches coalesce into the  subspace$\{(\pm z,z,0)\}$,  at a pitchfork.

To see what happens at $\left|\alpha/\beta\right|= 1$, we start  with $\left|\alpha/\beta\right|> 1$.
The submaximal branch exists when  
$\left|\Real\left(\alpha/\beta\right)\right|<1$ and hence  $\left| \Imag\left(\alpha/\beta\right)\right|>1$.
As $\left|\alpha/\beta\right|$ decreases, when we reach the value $1$
  we have 
$
\left|\alpha/\beta\right|^2=\left(\Real\left(\alpha/\beta\right)\right)^2+\left(\Imag\left(\alpha/\beta\right)\right)^2= 1
$
hence $\sin(2\phi)$ tends to $\pm\Imag\left(\alpha/\beta\right)$.
The expression for $r^2$ in \eqref{pitchfork condition} shows that in this case either $r$ tends to 0 or $r$ tends to $\infty$ and  the submaximal branches bifurcate from the conjugate subspaces   $\{(0,z,0)\}$ and $\{(z,0,0)\}$.

From the expression for $r^2$ in \eqref{pitchfork condition} it follows that the condition  $r=1$ at the submaximal branch only occurs when $\sin(2\phi)=0$ 
 and the conditions $r=0$, $r=\infty$ occur when $\left|\alpha/\beta\right|= 1$.
These conditions are satisfied at the values of $\alpha$ and $\beta$ for which the submaximal branch bifurcates from one of the $\CC$-axial subspaces, as we have already seen.
This implies that the submaximal solution branches  have no additional symmetry, because, for  any 
$\gamma$  in $\Gamma\times\Ss^1$,
 the norm of the first coordinate of $\gamma\cdot(\xi z,z,0)$ is either $|z|$, or $r|z|$, or zero. 
If  $r$ is neither $0$, $1$ or $\infty$, then the only possible symmetries are those that fix the subspace $\{(z_1,z_2,0)\}$.

\subsection{Submaximal branches in $\{(z_1,z_1,z_2)\}$}\label{secFixZ2kTil}
As a fixed-point subspace for the $\tetra\times\Ss^1$ action, this subspace  is  conjugate to 
$\fix\left(\ZZ_2\left( \kappa\right)\right)$.
 Under the $\OO\times\Ss^1$ action,
 it is  conjugate to $\fix\left(\ZZ_2\left(e^{\pi i}T^2C^2TC \right)\right)$
 and, under the $\cubo\times\Ss^1$ action,  
to $\fix\left(\ZZ_2\left(-T^2C^2TC \right)\right)$
(see Tables~\ref{table isotropy subgroups tetrahedral+k} and \ref{table isotropy subgroups}),  the conjugacy 
  in all cases being realised by  $C^2$. 
It contains the fixed-point subspaces 
 $\{(z,z,0)\}$, $\{(0,0,z)\}$, $\{(z,z,z)\}$,
corresponding to $\CC$-axial subgroups, as well as a conjugate copy of each one of them.
In what follows, we give a geometric  construction for finding solution branches of \eqref{normal form}  in the fixed-point subspace $\{(z_1,z_1,z_2)\}$, completing the description given in \cite{Ashwin}. 
In particular, we show that the existence of these branches only depends on the
 value of $\alpha/\beta$
 in the normal form, and that there are parameter values for which three submaximal solution branches coexist.

Restricting the  normal form \eqref{normal form} to this subspace and eliminating some solutions that lie in the two-dimensional fixed-point subspaces, one finds that   for $\beta\ne 0$ there is a solution branch through the point $(z,z,\xi z)$ with $\xi=re^{-i \psi}$, if and only if
\begin{equation}\label{firstReduction}
\beta-\alpha+r^2\alpha+\beta\left(r^2e^{-2i\psi}-2e^{2i\psi}\right)=0 
 \qquad\mbox{hence}\qquad
r^2=\frac{\alpha-\beta+2\beta e^{2i \psi}}{\alpha+\beta e^{-2i \psi} .
}
\end{equation}
 The expression for $r^2$ must be real and positive.
For $(x,y)=(\cos 2\psi,\sin 2\psi)$, this is equivalent to the conditions 
  $\rr(x,y)>0$ and $\ri(x,y)=0$ where:
$$
\rr(x,y)=\left(x-x_0\right)^2-\left(y-y_0\right)^2+K_r>0
\qquad
K_r=\left[\left(3\ \Imag\left(\alpha/\beta \right) \right)^2-
 \left(\Real\left(\alpha/\beta \right) +1\right)^2
\right]/16
$$
and
$$
\ri(x,y)=\left(x-x_0\right)\left(y-y_0\right)+K_i=0\qquad
K_i=3\ \Imag\left(\alpha/\beta \right)
\left[ \Real\left(\alpha/\beta \right)+1\right]/16
$$
 with
$$
x_0=\left(1-3 \Real\left(\alpha/\beta \right) \right)/4
\qquad
y_0=\Imag\left(\alpha/\beta \right)/4 
\quad\mbox{and}\quad
 r^2=2\rr(x,y)/\left|e^{-2i\psi} +\alpha/\beta\right|^2  .
$$
 Therefore, solution branches will correspond to points $(x,y)$ where the (possibly degenerate) hyperbola 
$\ri(x,y)=0$ intersects the unit circle, subject to the condition $\rr(x,y)>0$.
There may be up to four intersection points.
By inspection we find that $(x,y)=(1,0)$ is always an intersection point, where either $\psi=0$ or $\psi=\pi$.
Substituting  into \eqref{firstReduction}, it follows that $r=1$.
 Moreover, $\rr(1,0)=\left[(\Real(\alpha/\beta)+1)^2+(\Imag(\alpha/\beta))^2\right]/2\ge 0$.
These solutions correspond to the two-dimensional fixed point subspace $\{(z,z,z)\}$ and its conjugate $\{(z,z,-z)\}$.

Solutions that lie in the  fixed-point subspace $\{(z,z,0)\}$ correspond to $r=0$ and  hence to $\rr(x,y)=0$.
The subspace $\{(0,0,z)\}$ corresponds to $r\to\infty$ in \eqref{firstReduction}, i.e. to 
 $e^{-2i\psi}=-\alpha/\beta$. At these points both $\rr(x,y)$ and $\ri(x,y)$ are zero.

Generically, the fact that  $(1,0)$ is an intersection point implies that $\ri(x,y)=0$ meets  the circle on at least one more point. 
 The discussion above shows that all the other intersections that satisfy $\rr(x,y)>0$ correspond to submaximal branches.
We proceed to describe some of the situations that may arise.

\begin{figure}[ht]
\begin{center}
a) \includegraphics[scale=0.35]{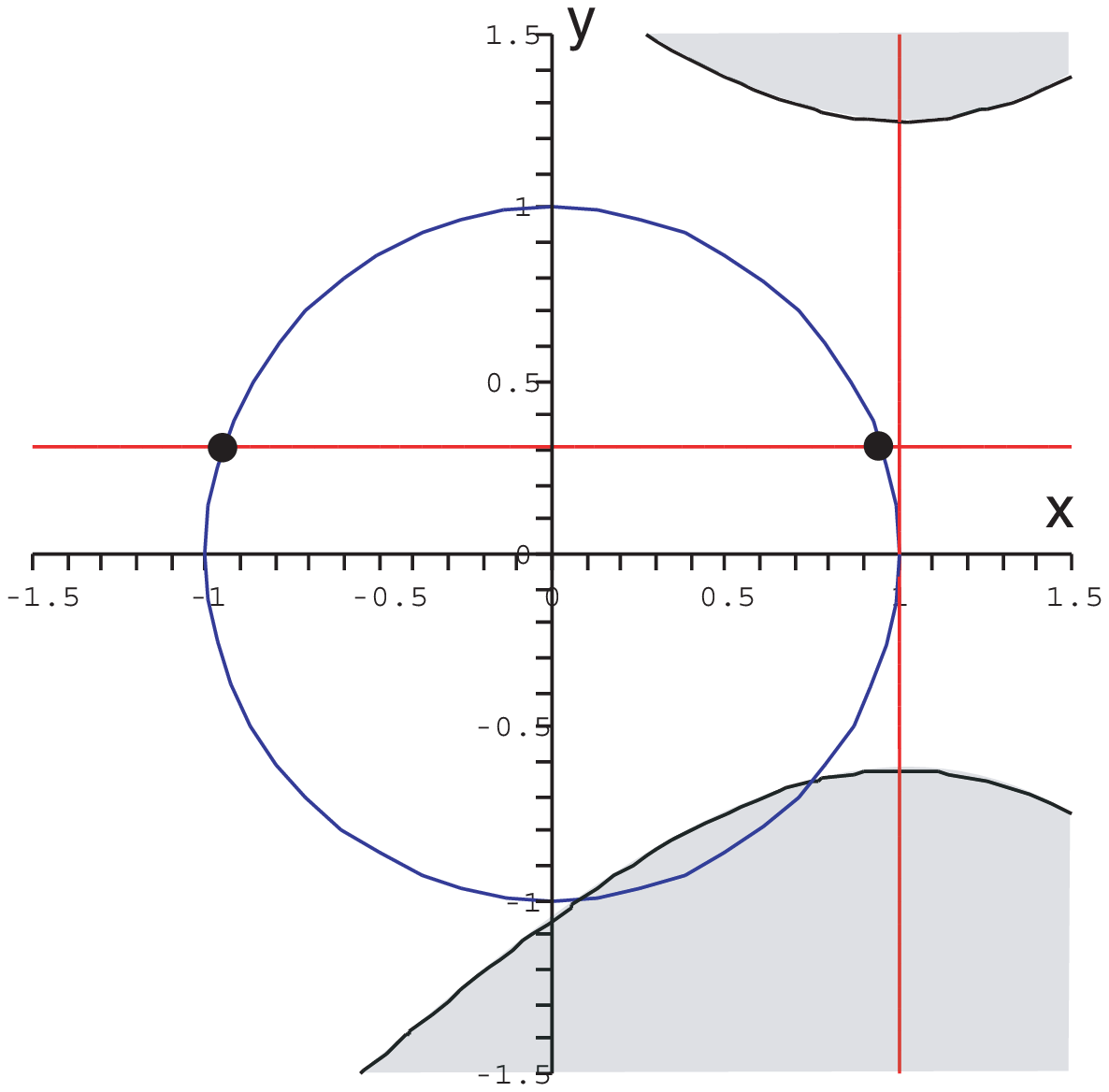}
 \qquad
b)  \includegraphics[scale=0.35]{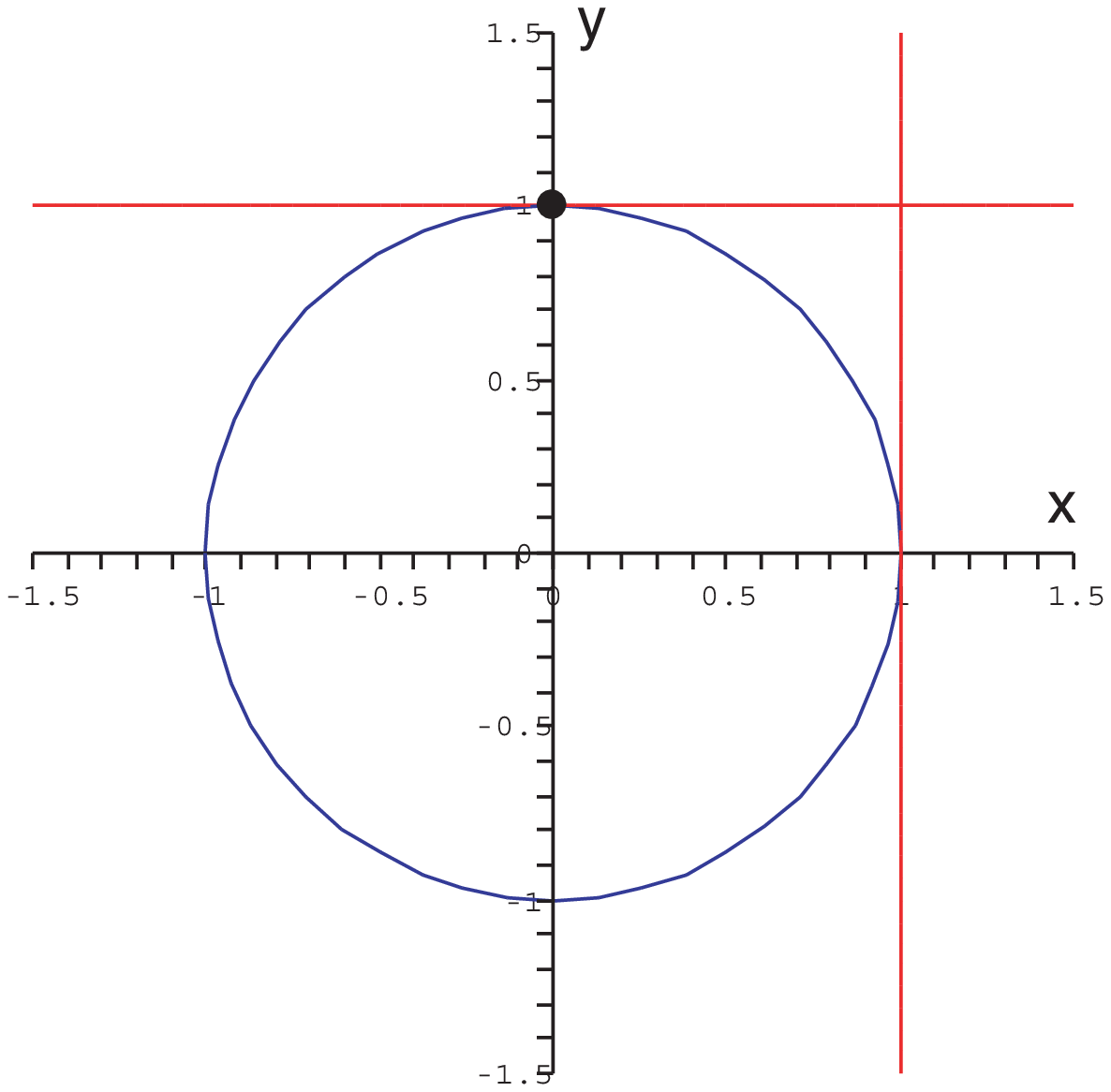}
\end{center}
\caption{
Graphs  of $\ri(x,y)=0$   in the $(x,y)$-plane, a degenerate hyperbola, consisting of two red lines, $x=1$ and $y=y_0$.
Submaximal solutions (black dots) arise where the degenerate hyperbola intersects the  unit circle (blue) with $\rr(x,y)>0$ (outside grey area).
The intersection at $(x,y)=(1,0)$ corresponds to a solution with maximal isotropy subgroup,  other intersections represent submaximal branches.
For $\alpha/\beta=-1+5i/4$ there are two submaximal solutions, shown in a). The solutions come together  when $\alpha/\beta=-1+4i$, shown in b).
}\label{figHypLines1}
\end{figure} 

\begin{figure}[ht]
\begin{center}
a)
 \includegraphics[scale=0.35]{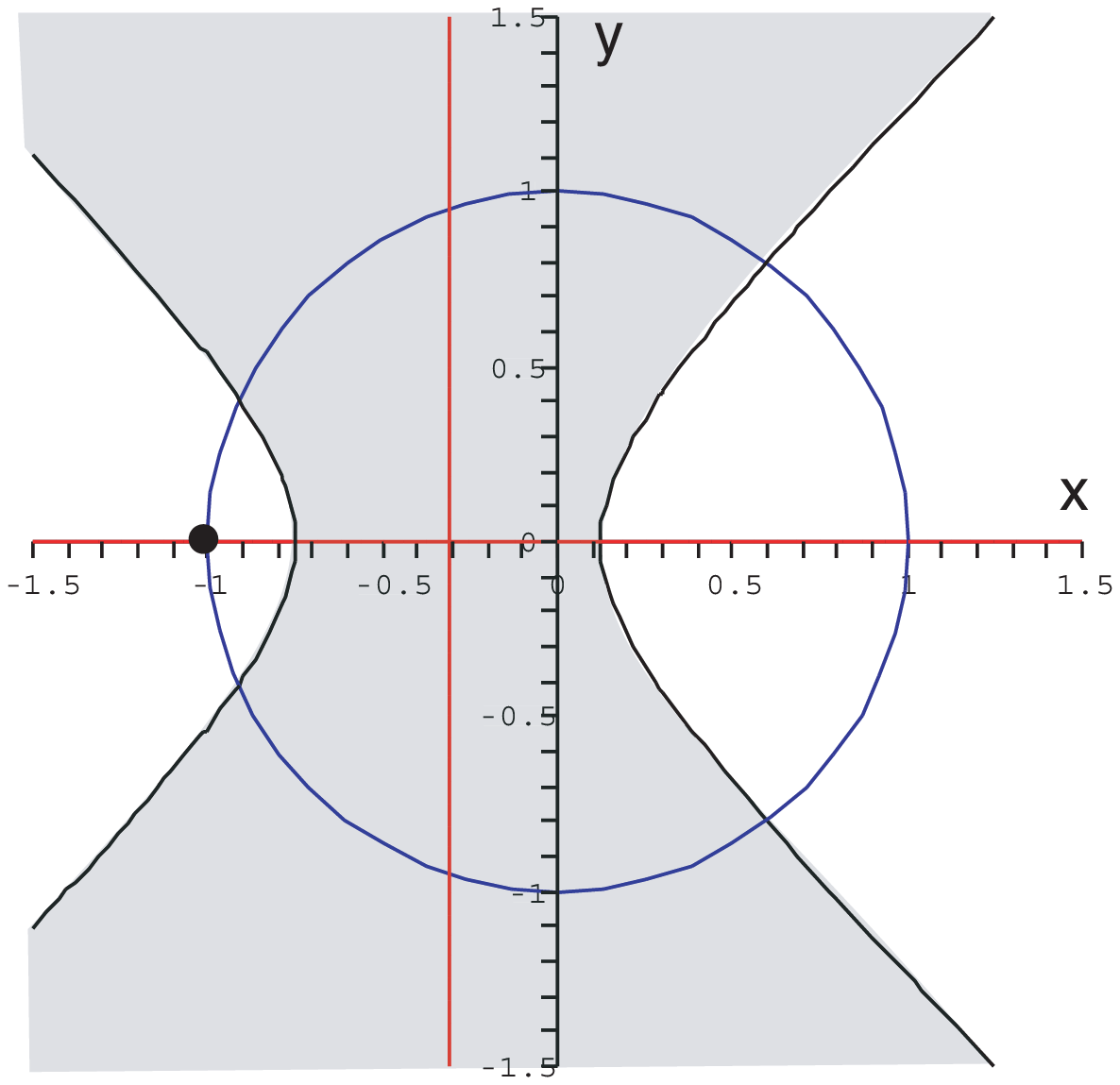}
 \qquad
 b)
  \includegraphics[scale=0.35]{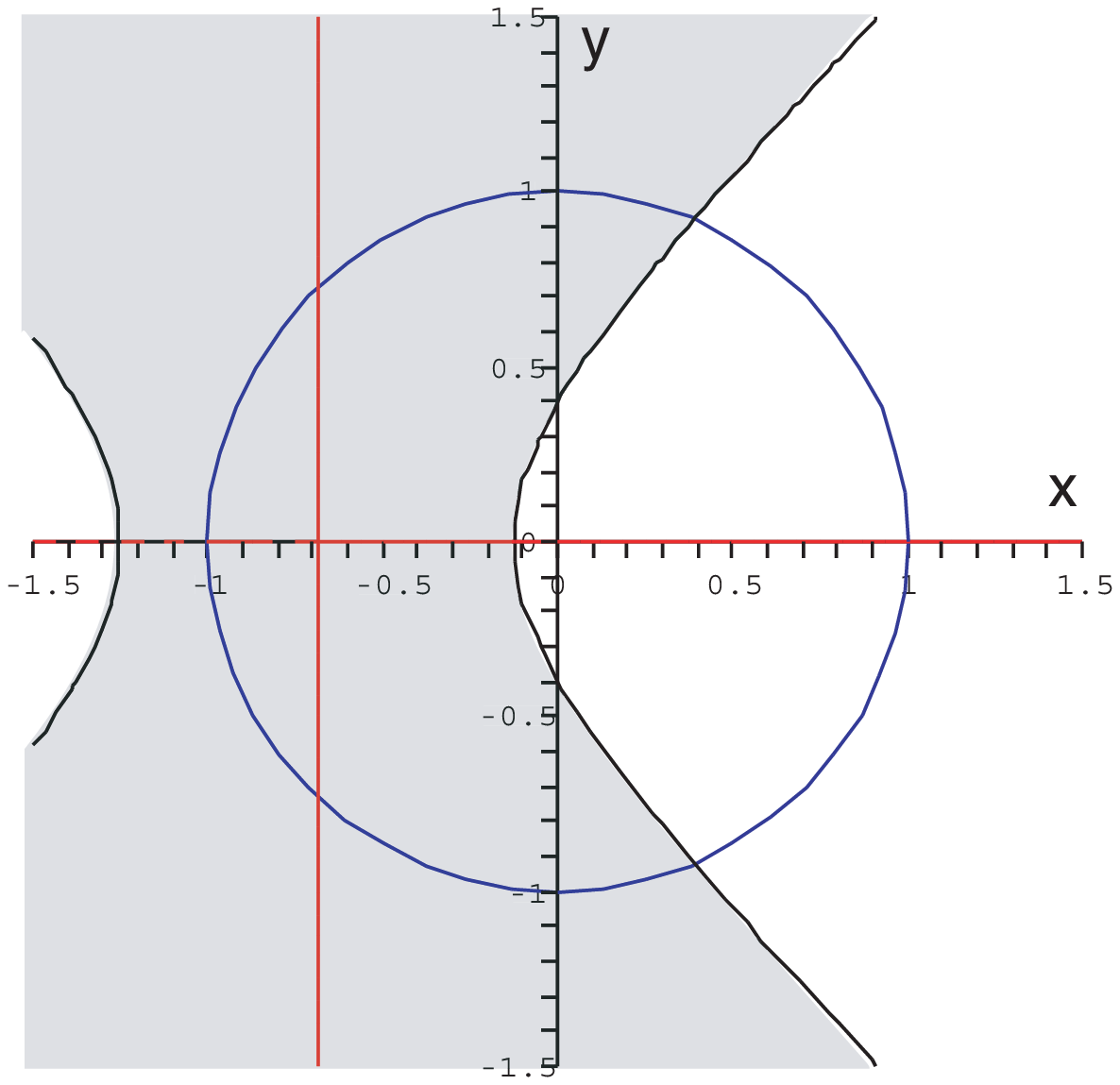}
\end{center}
\caption{Graphs  of $\ri(x,y)=0$  in the $(x,y)$-plane, a
degenerate hyperbola, the $y=0$ axis and the red line $x=x_0$ . Points  (black dots) where the degenerate hyperbola intersects the  unit circle (blue) with $\rr(x,y)>0$ (outside grey area) correspond to submaximal solutions.
The intersection at $(x,y)=(1,0)$ has maximal isotropy.
When $\Imag\left(\alpha/\beta\right)=0$ there is at most one intersection with $\rr(x,y)>0$ corresponding to a submaximal solution.
 Parameter values: 
 a) $\alpha/\beta=3/4$ (one  submaximal solution $(x,y)=  (-1,0)$), 
 b) $\alpha/\beta=5/4$ (no submaximal solution).}\label{figHypLines2}
\end{figure}

When $\Real\left(\alpha/\beta\right)=-1$, we have $K_i=0$ and hence $\ri(x,y)=0$ consists of
the two lines $x=1$ and $y=y_0$, shown in Figure~\ref{figHypLines1}. 
The line $x=1$ is tangent to the unit circle, hence, when $0<\left|\Imag\left(\alpha/\beta\right)\right|\le 4$,  the lines $\ri(x,y)=0$ intersect the circle at three points, except in the limit cases 
$\Imag\left(\alpha/\beta\right)=\pm 4$, when two solutions come together at a saddle-node.
The other case when $\ri(x,y)=0$ consists of  two lines 
%parallel to the axes 
occurs when $\Imag\left(\alpha/\beta\right)=0$, where $\ri(x,y)=0$ at $x=x_0$ or $y=y_0=0$, shown in 
Figure~\ref{figHypLines2}.
Then $\rr(x_0,y)<0$ and hence the only possible submaximal branch corresponds to $(x,y)=(-1,0)$.
Therefore, when the hyperbola $\ri(x,y)=0$ degenerates to two lines, there may be 0, 1 or 2 submaximal solution branches.

In the general case, when $\Real\left(\alpha/\beta\right)\ne -1$ and $\Imag\left(\alpha/\beta\right)\ne 0$ (Figures~\ref{figHyp1} and \ref{figHyp2}), the curve $\ri(x,y)=0$ is a hyperbola intersecting the unit circle at $(x,y)=(1,0)$ and at up to three other points. Again, these other intersections  may correspond to submaximal branches or not, depending on the sign of $\rr(x,y)$. 
Figure~\ref{figHyp1} shows examples with one branch and with no submaximal branches.
Examples with one, two and three submaximal branches are shown in Figure~\ref{figHyp2}.
These branches bifurcate from the fixed-point subspace $\{(z,z,0)\}$, when the intersection meets the line $\rr(x,y)=0$, or from the subspaces  $\{(z,z,\pm z)\}$. A pair  of branches may also terminate at a saddle-node bifurcation when the hyperbola is tangent to the unit circle.

\begin{figure}[ht]
\begin{center}
a)
 \includegraphics[scale=0.35]{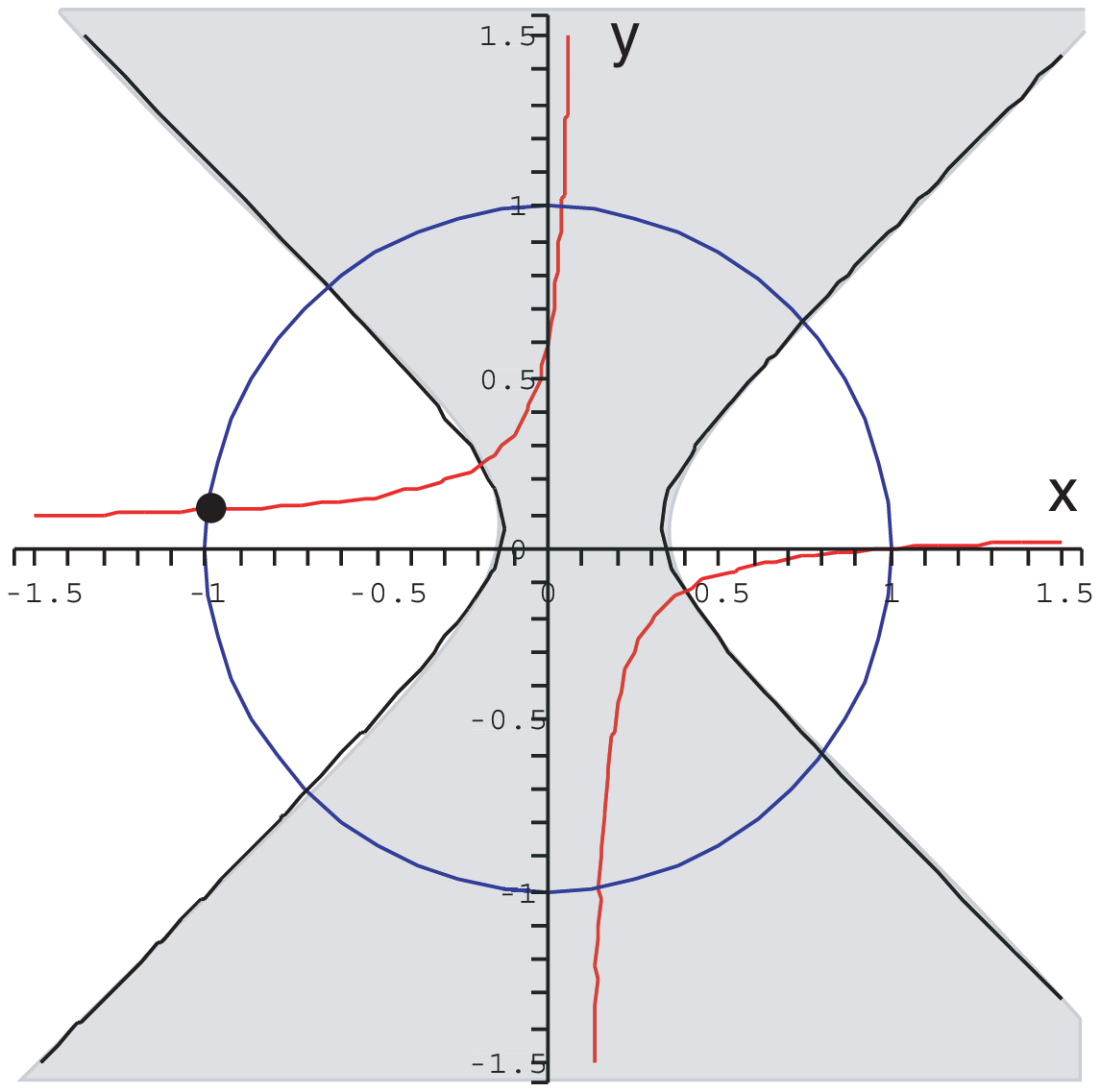}
 \qquad
 b)
  \includegraphics[scale=0.35]{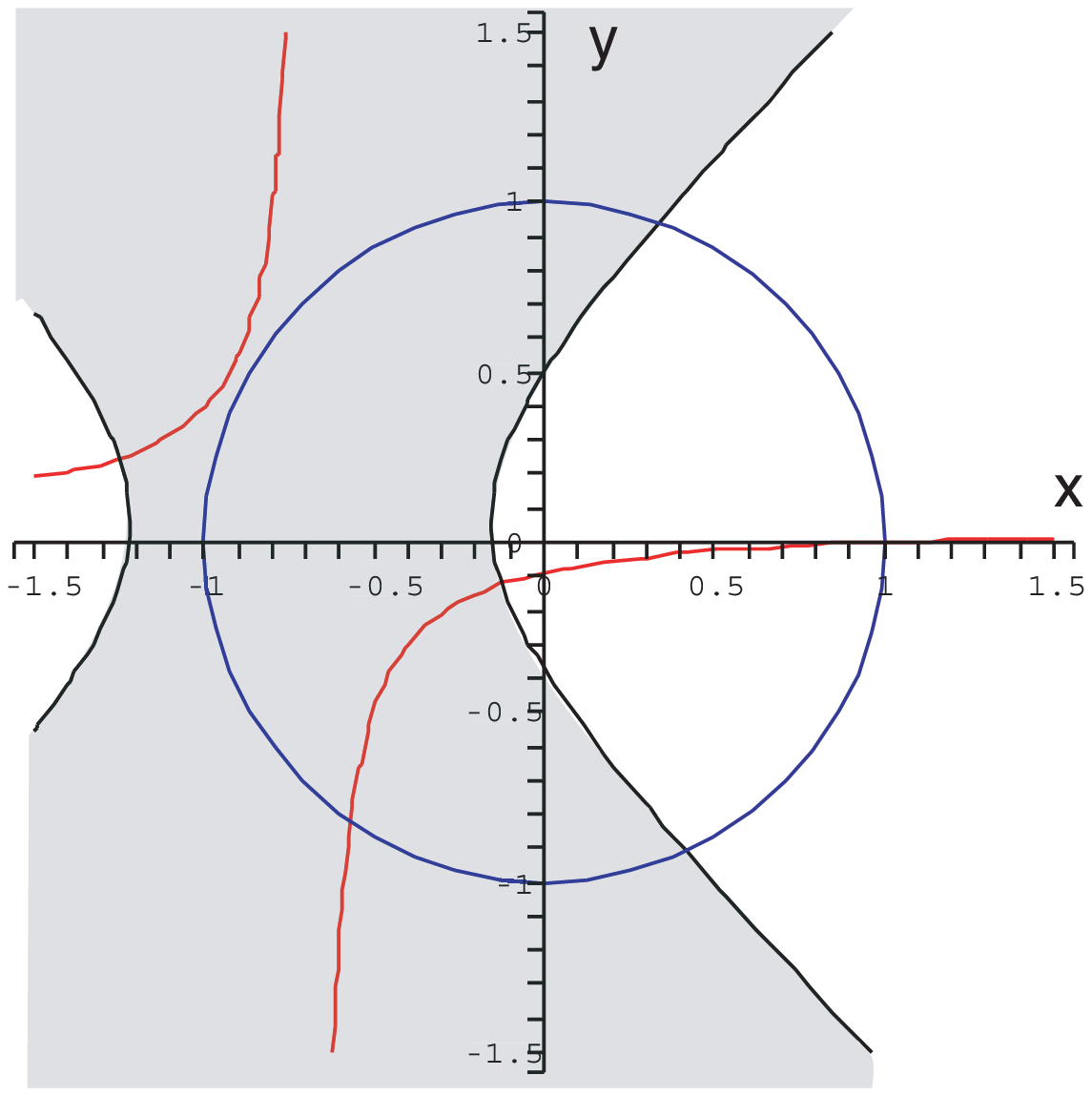}
\end{center}
\caption{Graphs  of $\ri(x,y)=0$ (red hyperbola) in the $(x,y)$-plane, for location of submaximal solutions (black dots), where the hyperbola intersects the  unit circle (blue) with $\rr(x,y)>0$  (outside grey area). 
The intersection at $(x,y)=(1,0)$ corresponds to a solution with maximal isotropy subgroup. Parameter values:
a) $\alpha/\beta=1/5+ i/4$ (one submaximal solution) and b) $\alpha/\beta=5/4+i/4$ (no submaximal solution).}
\label{figHyp1}
\end{figure} 

\begin{figure}[ht]
\begin{center}
a)
 \includegraphics[scale=0.30]{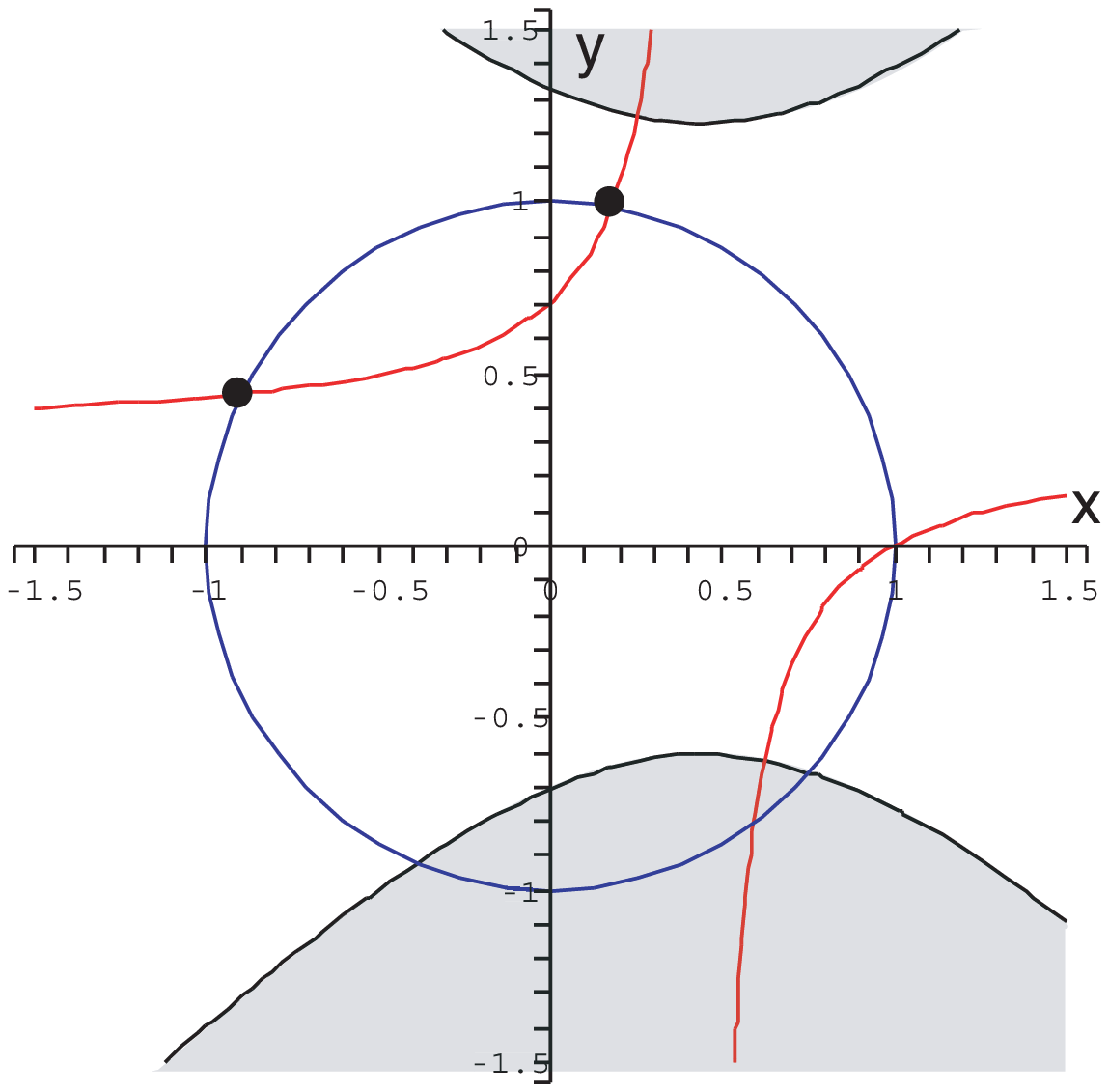}
 \  b)
  \includegraphics[scale=0.30]{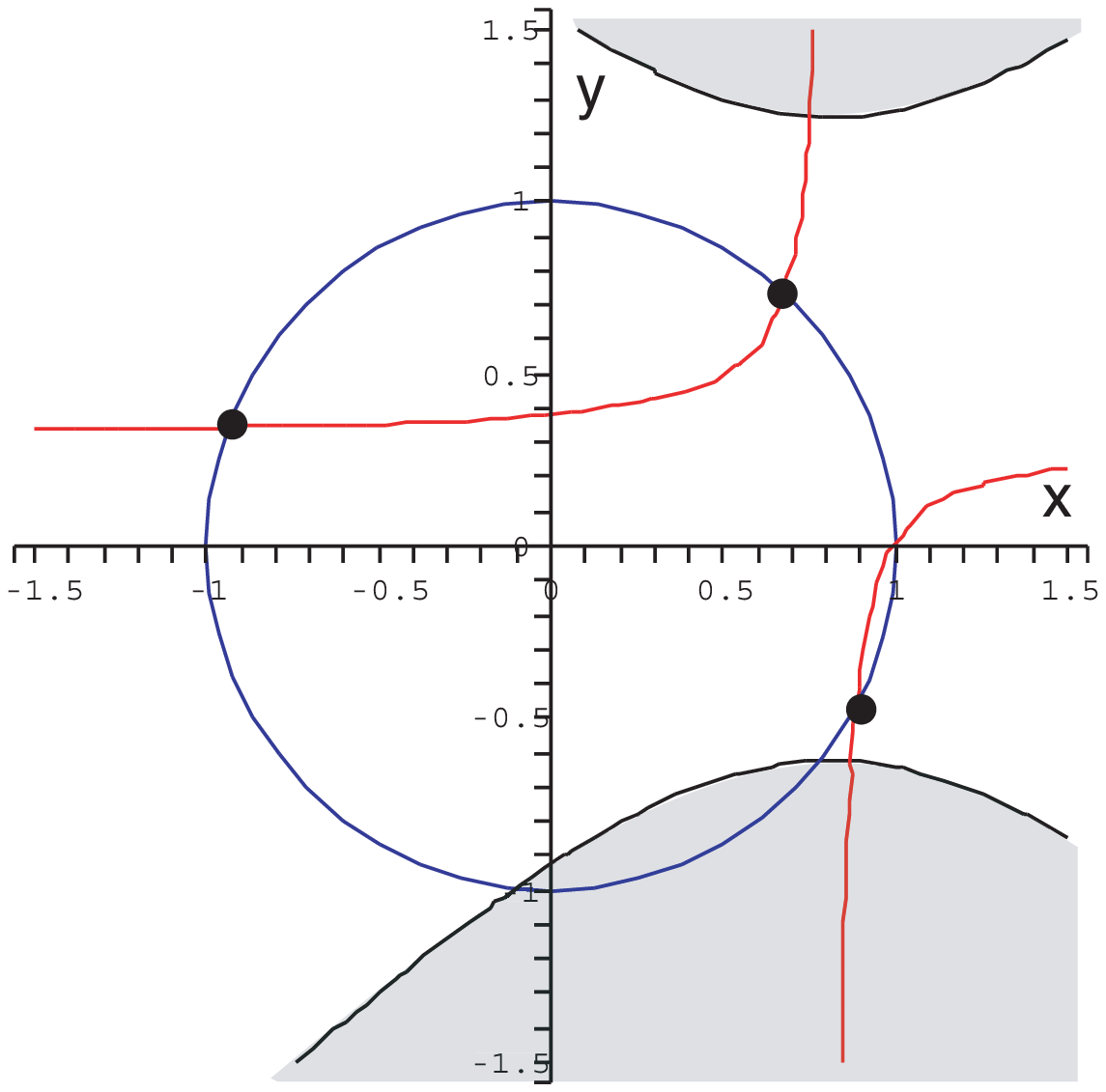} \qquad
\ c)
  \includegraphics[scale=0.30]{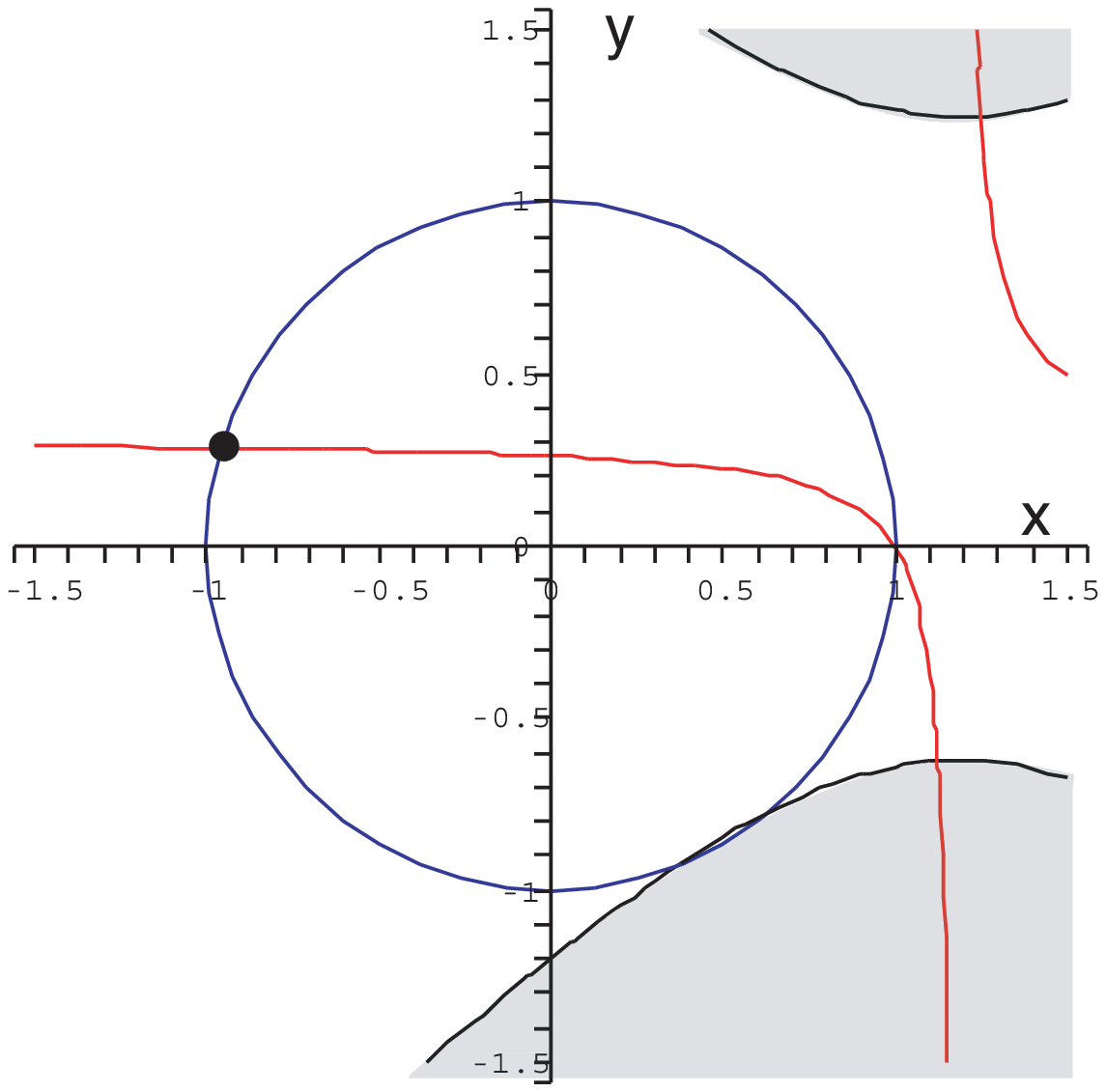}
\end{center}
\caption{Graphs  of $\ri(x,y)=0$ (red hyperbola)  in the $(x,y)$-plane, for location of submaximal solutions (black dots), where the hyperbola intersects the  unit circle (blue) with $\rr(x,y)>0$  (outside grey area). 
The intersection at $(x,y)=(1,0)$ corresponds to a solution with maximal isotropy subgroup.
Parameter values:
a) $\alpha/\beta=-1/4+ 5 i/4$ (two submaximal solutions), b) $\alpha/\beta=-3/4+ 5 i/4$ (three submaximal solutions), and c) $\alpha/\beta=-5/4+ 5 i/4$ (one submaximal solution).}
\label{figHyp2}
\end{figure}

\section{Spatio-temporal symmetries}\label{sectionSpatioTemporal}

The $H~\mathrm{mod}~K$ theorem \cite{Buono,GS03} states necessary and sufficient conditions  for the existence of a $\Gamma$-equivariant differential equation having a  periodic solution  with specified spatial symmetries $K\subset \Gamma$ and  spatio-temporal symmetries $H\subset \Gamma$, as explained in Section~\ref{secPreliminary}.

For a given $\Gamma$-equivariant differential equation, the $H~\mathrm{mod}~K$
Theorem
 gives necessary conditions on the  symmetries of periodic solutions.
Not all these solutions arise by a Hopf bifurcation from the trivial equilibrium --- we call this a \emph{primary Hopf bifurcation}.
In this section we address the question of determining which periodic
solution types, whose existence is guaranteed by the $H~\mathrm{mod}~K$ theorem, are obtainable at primary Hopf bifurcations, when the symmetry group $\Gamma$  is either $\tetra$ or $\OO$,  or $\cubo$.
The same branch predicted by the Hopf bifurcation theorem will have different spatial symmetries depending on the choice of $\Gamma$, and  hence, different dynamical interpretations.

The first step in  answering this question is the next lemma:

\begin{lema} Pairs of subgroups $H$, $K$ of symmetries of periodic solutions arising through a primary Hopf bifurcation  for $\Gamma=\tetra$ are given in Table~\ref{TableHopfTetra},
 for $\Gamma=\OO$ in Table~\ref{TableHopfCube},
 and for $\Gamma=\cubo$ in Table~\ref{TableHopfCube_I}.
\end{lema}
\begin{proof}
The symmetries corresponding to the $\CC$-axial subgroups of $\Gamma\times\Ss^1$ provide the first five rows of Tables~\ref{TableHopfTetra},  \ref{TableHopfCube}  and \ref{TableHopfCube_I}.
The last two rows correspond to the submaximal branches found in \ref{FixPiR} and  \ref{secFixZ2kTil} above.
\end{proof}

\begin{table}
\caption{Spatio-temporal symmetries of  solutions arising through primary Hopf bifurcation from the trivial equilibrium and number of branches, for the action of $\tetra\times\Ss^1$ on $\CC^3$.
The index refers to Table~\ref{table isotropy subgroups tetrahedral+k}.
Subgroups of $\tetra\times\Ss^1$ below the  dividing line are not  $\CC$-axial.}\label{TableHopfTetra}
\begin{tabular}{|c|c|c|c|c|}
\hline
index&subgroup of&Spatio-temporal &Spatial  &number \\
& $\Sigma\subset\tetra\times\Ss^1$&symmetries $H$ & symmetries $K$& of branches\\
&generators&generators& generators&\\
\hline\hline
(b)&$\{e^{\pi i}  C^2RC,\kappa\}$
&$\{C^2RC,\kappa\}$&$\{R,\kappa\}$&$3$\\
(c)&$\{C,\kappa\}$&$\{C,\kappa\}$&$\{C,\kappa\}$&$4$\\
(d)&$\{e^{-2\pi i/3}  C\}$&$\{C$\}&$\{Id\}$&$8$\\
(e)&$\{e^{\pi i}  R,\kappa \}$&$\{R,\kappa\}$&$\{\kappa\}$&$6$\\
(f)&$\{e^{\pi i/2}  C^2R\kappa\}$&$\{C^2R\kappa\}$&$\{Id\}$&$6$\\
\hline\hline
(g)&$\{e^{\pi i}R\}$&$\{R\}$&$\{Id\}$&12\\
(h)&$\{\kappa\}$&$\{\kappa\}$&$\{\kappa\}$&12\\
\hline
\end{tabular}
\end{table}

\begin{table}
\caption{Spatio-temporal symmetries of  solutions arising through primary Hopf bifurcation from the trivial equilibrium and number of branches,  for the action of $\OO\times\Ss^1$ on $\CC^3$.
The index refers to Table~\ref{table isotropy subgroups tetrahedral+k}.
Subgroups of $\OO\times\Ss^1$ below the  dividing line are not  $\CC$-axial.}\label{TableHopfCube}
\begin{tabular}{|c|c|c|c|c|}
\hline
index&subgroup of &Spatio-temporal &Spatial  &number \\
& $\Sigma\subset\OO\times\Ss^1$&symmetries $H$ & symmetries $K$& of branches\\
&generators&generators& generators&\\
\hline\hline
(b)&$\{TC^2,e^{\pi i}T^2\}$& $\{TC^2,T^2\}$   &$\{TC^2\}$&$3$\\
(c)&$\{C,e^{\pi i}T^2C^2T\}$& $\{C,T^2C^2T\}$   &$\{C\}$&$4$\\
(d)&$\{e^{-2\pi i/3}C\}$& $\{C\}$   &$\{Id\}$&$8$\\
(e)&$\{e^{\pi i}TC^2TC^2,T^3C^2\}$&$\{TC^2TC^2,T^3C^2\}$  &$\{T^3C^2\}$&$6$\\
(f)&$\{e^{-\pi i/2}T\}$&   $\{T\}$       &$\{Id\}$&$6$\\
\hline\hline
(g)&$\{e^{\pi i}TC^2TC^2\}$&$\{TC^2TC^2\}$&$\{Id\}$&12\\
(h)&$\{e^{\pi i}T^2C^2TC\}$&$\{T^2C^2TC\}$&$\{Id\}$&12\\
\hline
\end{tabular}
\end{table}

\begin{table}
\caption{Spatio-temporal symmetries of  solutions arising through primary Hopf bifurcation from the trivial equilibrium and number of branches,  for the action of $\cubo\times\Ss^1$ on $\CC^3$.
The index refers to Table~\ref{table isotropy subgroups tetrahedral+k}.
Subgroups of $\cubo\times\Ss^1$ below the  dividing line are not  $\CC$-axial.}\label{TableHopfCube_I}
\begin{tabular}{|c|c|c|c|c|}
\hline
index&subgroup of&Spatio-temporal &Spatial  &number \\
& $\Sigma\subset\cubo\times\Ss^1$&symmetries $H$ & symmetries $K$& of branches\\
&generators&generators& generators&\\
\hline\hline
(b)&$\{TC^2,-T^2\}$& $\{TC^2,-T^2\}$   &$\{TC^2,-T^2\}$&$3$\\
(c)&$\{C,-T^2C^2T\}$& $\{C,-T^2C^2T\}$   &$\{C,-T^2C^2T\}$&$4$\\
(d)&$\{e^{-2\pi i/3}C\}$& $\{C\}$   &$\{Id\}$&$8$\\
(e)&$\{-TC^2TC^2,T^3C^2\}$&$\{-TC^2TC^2,T^3C^2\}$  &$\{-TC^2TC^2,T^3C^2\}$&$6$\\
(f)&$\{e^{-\pi i/2}T\}$&   $\{T\}$       &$\{Id\}$&$6$\\
\hline\hline
(g)&$\{-TC^2TC^2\}$&$\{-TC^2TC^2\}$&$\{-TC^2TC^2\}$&24\\
(h)&$\{-T^2C^2TC\}$&$\{-T^2C^2TC\}$&$\{-T^2C^2TC\}$&24\\
\hline
\end{tabular}
\end{table}

The next step is to identify the subgroups corresponding to Theorem~\ref{teorema H mod K}.

\begin{table}
\begin{center}
\caption{Possible pairs $H$, $K$ for Theorem~\ref{teorema H mod K} in the action of $\tetra$ on
 $\CC^3$.}\label{tableTetraHK}
\end{center}
\begin{tabular}{|c|c|c|c|c|c|}
\hline
$K$&Generators of $K$&$H$&Generators of $H$&$\fix(K)$&dim\\
\hline\hline
$\mathbb{D}_2$&$\{R,\kappa\}$&
$K$&$\{R,\kappa\}$&$\{(z,0,0)\}$&$2$  \\
&&$N(K)=\mathbb{D}_4$&$\{C^2RC,\kappa\} $&&\\
\hline
$\mathbb{D}_3$ &$\{C,\kappa\}$&
$N(K)=K$&{$\{C,\kappa\}$}&$\{(z,z,z)\}$&$2$\\
\hline
$\ZZ_2$&$\{\kappa\}$&$K$&$\{\kappa\}$&$\{(z_1,z_2,z_2)\}$&$4$\\
&&$N(K)=\mathbb{D}_2$&$\{R,\kappa\}$&&\\
\hline
$\one$&$\{Id\}$&$\ZZ_4$&$\{C^2R\kappa\}$&$\CC^3$&6\\
&&$\ZZ_3$&$\{C\}$&&\\
&&$\ZZ_2$&$\{R\}$&&\\
&&$\ZZ_2$&$\{\kappa\}$&&\\
&&$K$&$\{Id\}$&&\\
\hline
\end{tabular}
\end{table}

\begin{lema}\label{lema pairs H K}
Pairs of  subgroups $H,K$ satisfying conditions $(a)-(d)$ of Theorem \ref{teorema H mod K} are given in 
Table~\ref{tableTetraHK} for $\Gamma=\tetra$,
in Table~\ref{tableOctaHK} for $\Gamma=\OO$
 and in Table~\ref{tableOcta-1} for $\Gamma=\cubo$.
\end{lema}

\begin{proof}
Conditions $(a)-(d)$ of Theorem~\ref{teorema H mod K} are immediate for the isotropy subgroups with two-dimensional fixed-point subspaces. For $\Gamma=\OO$, these are all the non-trivial subgroups.

 Condition $(d)$ has to be verified for the isotropy subgroups $K$ of $\tetra$ and of  $\cubo$ that have a four-dimensional fixed-point subspace, as well as for $K=\one$ for the three groups $\Gamma$.
Since,  for each $\gamma\in\Gamma$, we have that $\dim\fix(\gamma)$ is even, then  $L_K$ is always the union of a finite number of subspaces with even real  dimension.
Therefore, $L_K$ has even codimension in $\fix(K)$ and hence
$\fix(K)\backslash L_K$ is connected and condition $(d)$ follows  in all cases.

For $\Gamma=\cubo$, condition $(a)$ has to be verified for the two subgroups with four-dimensional fixed-point subspaces.
For $K=\ZZ_2=\langle -T^2C^2TC\rangle$ we have 
$N\left(K\right)=\langle  T^2C^2TC,T^3C^2,-I \rangle$ a subgroup of order 8.
All elements of $N\left(K\right)$,
except for $Id$, have order 2, hence $N\left(K\right)/K$ is not cyclic. This gives rise to 3 non-conjugate possibilities for $H$, all isomorphic to $\mathbb{D}_2$.

For $K=\widehat{\ZZ}_2=\langle -TC^2TC^2\rangle$ we have 
$N(\widehat{\ZZ}_2)=\langle TC^2, T^2,-I \rangle$,  a subgroup of order 16,
where $TC^2$ has order 4, $T^2$ and $-Id$ have order 2, hence $N(\widehat{\ZZ}_2)/\widehat{\ZZ}_2$ is not cyclic. This gives rise to 3 non-conjugate possibilities for $H$,  one  isomorphic to $\mathbb{D}_4$, the others isomorphic to $\mathbb{D}_2$.

Condition $(a)$ also has to be used for the subgroup $K=\one$ in all cases, where the possibilities for $H$ are then all the cyclic subgroups of $\Gamma$. We omit this information in Table~\ref{tableOcta-1}, since there are too many subgroups in this case.
\end{proof}

Of all the  subgroups of $\tetra$ of order two that appear as $H$ in a pair $H\sim\ZZ_n,K=\one$
in Table~\ref{tableTetraHK}, only $H\sim\ZZ_3(C)$ is an isotropy subgroup, as can be seen in 
Figure~\ref{FigLatticeTetra2}.
The two subgroups of $\tetra$ of order two, $H=\ZZ_2(\kappa)$ and $H=\ZZ_2(R)$ are not  conjugate, since they are generated by the image by $\rho_1$ of the non conjugate elements $(34)$ and $(12)(34)$ of $S_4$, respectively.
Another way to see this is to observe that 
$\kappa$ fixes a four-dimensional subspace, whereas $R$ fixes a subspace of dimension two. 
In contrast, all but one of the cyclic subgroups of $\OO$ are isotropy subgroups, as can be seen comparing Table~\ref{tableOctaHK} to Figure~\ref{FigLatticeOcta},  and  noting that $\ZZ_4(T)$ is  conjugate to $\ZZ_4(TC^2)$ and that $\ZZ_2(T^2C^2TC)$ is  conjugate to $\ZZ_2(T^3C^2)$.
This will have a marked effect on the primary Hopf bifurcations.

\begin{table}
\caption{{ Possible pairs $H$, $K$ for Theorem~\ref{teorema H mod K} in the action of $\OO$ on $\CC^3$. }}\label{tableOctaHK}
\begin{tabular}{|c|c|c|c|c|c|}
\hline
$K$&Generators of $K$&$H$&Generators of $H$
&$\fix(K)$&$\mathrm{dim}$\\
\hline\hline
$\ZZ_4$&$\{TC^2\}$&{ $K$}&$\{TC^2\}$&$\{(z,0,0)\}$&$2$  \\
&& $N(K)=\mathbb{D}_4$&$\{T^2,TC^2\}$&&  \\
\hline
$\ZZ_3$&$\{C\}$&{ $K$}&$\{C\}$&$\{(z,z,z)\}$&$2$  \\
&& $N(K)=\mathbb{D}_3$&$\{C,T^2C^2T\}$&&  \\
\hline
$\ZZ_2$ &$\{T^3C^2\}$&{ $K$}&$\{T^3C^2\}$&$\{(0,z,z)\}$&$2$  \\
&& $N(K)=\mathbb{D}_2$&$\{TC^2TC^2,T^3C^2\}$&&  \\
\hline
$\one$&$\{Id\}$&$\ZZ_4$&$\{T\}$&$\CC^3$&6\\
&&$\ZZ_3$&$\{C\}$&&\\
&&$\ZZ_2$&$\{TC^2TC^2\}$&&\\
&&$\ZZ_2$&$\{T^2C^2TC\}$&&\\
&&$K$&$\{Id\}$&&\\
\hline
\end{tabular}
\end{table}

\begin{table}
\caption{{ Possible pairs $H$, $K$, with $K\ne\one$ for Theorem~\ref{teorema H mod K} in the action of $\cubo$ on $\CC^3$. }}\label{tableOcta-1}
\begin{tabular}{|c|c|c|c|c|c|}
\hline
$K$&Generators of $K$&$H$&Generators of $H$
&$\fix(K)$&$\mathrm{dim}$\\
\hline\hline
$\mathbb{D}_4$&$\{TC^2,-T^2\}$& $K$&$\{TC^2,-T^2\}$&$\{(z,0,0)\}$&$2$  \\
&&$N(K)=\langle K,- Id\rangle$&$\{TC^2,T^2,- Id\}$&& \\
\hline
$\mathbb{D}_3$&$\{C,-T^2C^2T\}$&$K$&$\{C,-T^2C^2T\}$&$\{(z,z,z)\}$&$2$  \\
&&$N(K)=\langle K,- Id\rangle$&$\{C,T^2C^2T,- Id\}$&&  \\
\hline
$\mathbb{D}_2$ &$\{T^3C^2,-T^2C^2TC^2\}$&{ $K$}&$\{T^3C^2,-T^2C^2TC^2\}$&$\{(0,z,z)\}$&$2$  \\
&& $N(K)=\langle K,- Id\rangle$&$\{T^3C^2,T^2C^2TC^2,- Id\}$&&  \\
\hline
$\ZZ_2$&$\{-T^2C^2TC\}$&$K$&$\{-T^2C^2TC\}$&$\{(z_1,z_2,z_2)\}$&4\\
&&$\ZZ_2\oplus\ZZ_2(- Id)$&$\{T^2C^2TC,- Id\}$&&\\
%repetida, igual ao grupo da linha seguinte:
%&&$\ZZ_2\oplus\ZZ_2(T^3C^2)$&$\{-T^2C^2TC,T^3C^2\}$&&\\
&&$\ZZ_2\oplus\widehat{\ZZ}_2$&$\{-T^2C^2TC,TC^2TC^2\}$&&\\
\hline
$\widehat{\ZZ}_2$&$\{-TC^2TC^2\}$&$K$&$\{-TC^2TC^2\}$&$\{(0,z_1,z_2)\}$&4\\
&&$\ZZ_4(TC^2)\oplus\ZZ_2(-Id)$&$\{- Id,TC^2\}$&&\\
&&$\widehat{\ZZ}_2\oplus\ZZ_2(- Id)$&$\{- Id,TC^2TC^2\}$&&\\
&&$\widehat{\ZZ}_2\oplus\ZZ_2(T^2)$&$\{-TC^2TC^2,T^2\}$&&\\
&&$\widehat{\ZZ}_2\oplus\ZZ_2(-T^2)$&$\{-TC^2TC^2,-T^2\}$&&\\
&&$\widehat{\ZZ}_2\oplus\ZZ_2(T^2C^2TC)$&$\{-TC^2TC^2,T^2C^2TC\}$&&\\
&&$\ZZ_2\oplus\widehat{\ZZ}_2$&$\{-TC^2TC^2,-T^2C^2TC\}$&&\\
\hline
\end{tabular}
\end{table}

\begin{prop}\label{propPrimaryTetra}
For the representation of the group $\tetra$ on $\RR^6\sim\CC^3$, all 
pairs of  subgroups $H,K$ satisfying the conditions of  the $H~\mathrm{mod}~K$ Theorem,  with $H\ne\one$, occur as spatio-temporal symmetries of periodic solutions arising through a primary Hopf bifurcation, except for:
\begin{enumerate}
\item\label{HKD2}
the  pair $H=K=\mathbb{D}_2$, generated by $\{R,\kappa\}$;
\item\label{HZ2K1}
the pair $H=\ZZ_2(\kappa)$, $K=\one$.
\end{enumerate}
\end{prop}

\begin{proof}
The result follows by inspection of Tables~\ref{TableHopfTetra} and \ref{tableTetraHK}, we discuss here why these pairs do not arise in a primary Hopf bifurcation.
Case \eqref{HKD2} refers to a non-trivial isotropy subgroup $K\subset\tetra$ for which $N(K)\ne K$.
For the group $\tetra$, there are two non-trivial isotropy subgroups in this situation,  as can be seen in 
Table~\ref{tableTetraHK}.
The first one, $K=\mathbb{D}_2$, is  not an isotropy subgroup of $\tetra\times\Ss^1$,
so the pair $H=K=\mathbb{D}_2$ does not occur in Table~\ref{TableHopfTetra} as a Hopf bifurcation from the trivial solution in a normal form with symmetry $\tetra$.
The second subgroup, $K=\ZZ_2(\kappa)$ occurs as an isotropy subgroup  of $\tetra\times\Ss^1$ with four-dimensional fixed-point subspace.
As we have seen in \ref{secFixZ2kTil}, in this subspace there are periodic solutions with $K=H=\ZZ_2(\kappa)$ arising through a Hopf bifurcation with submaximal symmetry. 
 On the other hand, the normaliser of $\ZZ_2(\kappa)$ corresponds to a $\CC$-axial subgroup of $\tetra\times\Ss^1$, so there is a Hopf bifurcation from the trivial solution with $H=N(K)$. 

Case \eqref{HZ2K1} concerns the situation when $K=\one$. 
All the cyclic subgroups $H\subset\tetra$, with the exception of $\ZZ_2(\kappa)$,  are the projection into $\tetra$ of cyclic isotropy subgroups of $\tetra\times\Ss^1$, so they correspond to primary Hopf bifurcations. 
\end{proof}

Another reason why the   pair $H=K=\mathbb{D}_2$  does not occur at a primary Hopf bifurcation is the following:  a non-trivial periodic solution in $\fix(\mathbb{D}_2)$ has the form $X(t)=(z(t),0,0)$.
Then $Y(t)=C^2RCX(t)=-X(t)$ is also a solution contained in the same plane. If the origin is inside $X(t)$ then the curves $Y(t)$ and $X(t)$ must intersect, so they coincide as curves, and this means that $C^2RC$ is a spatio-temporal symmetry of   $X(t)$.
Hence, if $H=K=\mathbb{D}_2$, then $X(t)$ cannot encircle the origin, and hence it cannot arise from a Hopf bifurcation from the trivial equilibrium. Of course it may originate at a Hopf bifurcation from another equilibrium.
The argument does not apply to the other subgroup $K$ with $N(K)\ne K$, because in this case 
$\dim\fix(K)=4$ and indeed, for some values of the parameters in the normal form,  there are primary Hopf bifurcations into solutions whith $H=K=\ZZ_2(\kappa)$.

 The second case in Proposition~\ref{propPrimaryTetra} is more interesting: if $X(t)=(z_1(t),z_2(t),z_3(t))$ is a $\rho$-periodic solution of a $\tetra$-invariant differential equation, with $H=\ZZ_2(\kappa)$, $K=\one$, then for some $\theta\ne 0 \pmod{\rho}$ and for all $t$ we have $z_1(t+\theta)=z_1(t)$ and $z_3(t+\theta)=z_2(t)=z_2(t+2\theta)$.

If $z_2(t)$ and $z_3(t)$ are identically zero, then  $\kappa\in K$, contradicting our assumption.
If $z_1(t)\equiv 0$ and $z_2(t)$ and $z_3(t)$ are both non-zero, then $X(t)$  cannot be obtained from the $\tetra\times\Ss^1$ action, because in this case it would be $X(t)\in\fix(\ZZ_2(e^{\pi i}R))$, hence $R\in H$.
The only other possibility is that  to have all coordinates of $X(t)$ not zero, 
$X(t)=(z_1(t),z_2(t),z_2(t+\theta))$ with $2\theta=\rho$, 
with a $\theta$-periodic $z_1$ and $z_2(t)$, and $z_3(t)$ having twice the period of $z_1(t)$.
In this case, if $X(t)$ would bifurcate from an equilibrium in the normal form \eqref{normal form}, at the bifurcation point the eigenvalues would be $\pm 2\pi i/\theta$ and $\pm \pi i/\theta$, a 2-1 resonance.
This last situation is not compatible with  the normal form \eqref{normal form}, where the linearisation at the origin is a (complex) multiple of the identity.
%This last situation is not forced to occur for the group $\OO$, as can be seen in the next result.

\begin{prop}\label{propPrimaryOcta}
For the representation of the group $\OO$ on $\RR^6\sim\CC^3$, all 
pairs of  subgroups $H,K$ satisfying the conditions of  the $H~\mathrm{mod}~K$ Theorem,
 with $H\ne\one$, occur as spatio-temporal symmetries of periodic solutions arising through a primary Hopf bifurcation, except for the following pairs:
\begin{enumerate}
\item\label{Z4}
$H=K=\ZZ_4$, generated by $\{TC^2\}$; 
\item\label{Z3}
$H=K=\ZZ_3$ generated by $\{C\}$; 
\item\label{Z2}
 $H=K=\ZZ_2$ generated by $\{T^3C^2\}$.
\end{enumerate}
\end{prop}

\begin{proof}
As in Proposition~\ref{propPrimaryTetra}, the result follows comparing Tables~\ref{TableHopfCube} and \ref{tableOctaHK}.

All the cyclic subgroups $H\subset\OO$ are the projection into $\OO$ of cyclic isotropy subgroups $\Sigma$ of $\OO\times\Ss^1$, so the pairs $H,K=\ZZ_n,\one$ correspond to primary Hopf bifurcations:  for $\ZZ_4$ and $\ZZ_3$  the subgroup $\Sigma\subset\OO\times\Ss^1$ is $\CC$-axial, whereas for the  subgroups of $\OO$ of order two,  the subspace $\fix(\Sigma)$ is four-dimensional and has been treated in \ref{FixPiR} and \ref{secFixZ2kTil}  above.

All the non-trivial isotropy subgroups $K$ of $\OO$ satisfy $N(K)\ne K$, so they are candidates for cases where $H=K$ does not occur. This is indeed the case, since they are not isotropy subgroups of $\OO\times\Ss^1$.
\end{proof}

The pairs $H,K$ in Proposition~\ref{propPrimaryOcta},
that are not symmetries of solutions arising through primary Hopf bifurcation, are of the form $H=K$, where there would be only spatial symmetries. 
Here, $\dim\fix(K)=2$ for all cases. 
As remarked after the proof of Proposition~\ref{propPrimaryTetra}, the origin cannot lie inside a
closed trajectory with these symmetries, hence they cannot arise at a  Hopf bifurcation from the trivial equilibrium.
%at a primary Hopf bifurcation --- a bifurcation from the trivial equilibrium in the normal form \eqref{normal %form}. 
In a $\OO$-equivariant differential equation they may only bifurcate from a non-trivial equilibrium.

For the larger group $\cubo$ there are even more possibilities for the $H~\mathrm{mod}~K$ Theorem, and hence, even more of them will not occur as primary Hopf bifurcations:

\begin{prop}\label{propPrimaryCube}
For the representation of the group $\cubo$ on $\RR^6\sim\CC^3$, 
 there are 11 pairs $H,K$ with $K\ne\one$ satisfying the conditions of  the $H~\mathrm{mod}~K$ Theorem,  that do not occur as spatio-temporal symmetries of periodic solutions arising through a primary Hopf bifurcation.
\end{prop}

\begin{proof}
As in Proposition~\ref{propPrimaryTetra}, the result follows comparing Tables~\ref{TableHopfCube_I} and \ref{tableOcta-1}. Since $-Id$ is in the normaliser of every subgroup of $\cubo$ and does not belong to any isotropy subgroup, this increases dramatically the number of possible pairs $H,K$.
Here, all the pairs of  subgroups $H,K$ that occur as spatio-temporal symmetries of periodic solutions arising through a primary Hopf bifurcation, also 
 satisfy $H=K$, except for the pairs $H=\langle C\rangle$,  $K=\one$ and $H=\langle T\rangle$, $K=\one$.
\end{proof}

Note that the discussion of the pair $H=\ZZ_2(\kappa)$, $K=\one$ after Proposition~\ref{propPrimaryTetra} applies here to the pair $H=\ZZ_2(-T^2C^2TC)$, $K=\one$ since $\kappa=-T^2C^2TC$.

When a cyclic  subgroup of $\Gamma$ occurs as $H$ paired  with $K\ne\one$ for the $H~\mathrm{mod}~K$ Theorem, then 
  it may also  occur  paired with the trivial subgroup $\one$.
This is the case of $H=\langle \kappa\rangle$, for $\Gamma=\tetra$; 
of $H=\langle TC^2\rangle$ ( conjugate to $\langle T\rangle$), $H=\langle C\rangle$,  
and $H=\langle T^3C^2\rangle$ ( conjugate to $\langle TC^2TC^2\rangle$), for $\Gamma=\OO$;
and of $H=\ZZ_2(-T^2C^2TC)$, and 
$H=\widehat{\ZZ}_2=\langle-TC^2TC^2\rangle$ for $\Gamma=\cubo$.
 In all these cases, there may be periodic solutions with the same spatio-temporal symmetry and less spatial symmetry. In all these cases, only one of the two pairs occurs at a primary Hopf bifurcation.

For the group $\cubo$ there is also a non-trivial case where the same subgroup $H$ is paired with different subgroups  $K$,
as can be seen in Table~\ref{tableOcta-1}: the subgroup
$H=\ZZ_2(-T^2C^2TC^2)\oplus\widehat{\ZZ}_2$ occurs both with $K=\ZZ_2(-T^2C^2TC^2)$ and with $K=\widehat{\ZZ}_2$. In this case, none of the possibilities  occurs at a primary Hopf bifurcation.

\subsection*{ Acknowledgements}
We would like to thank the anonymous referees for helpful comments that have greatly improved the article, and in particular, for suggesting the addition of the analysis of the full symmetry group of the cube.
We also thank Ana Paula S. Dias for helpful discussions.
The research of both authors at Centro de Ma\-te\-m\'a\-ti\-ca da Universidade do Porto (CMUP)
 had financial support from
 the European Regional Development Fund through the programme COMPETE and
 from  the Portuguese Government through the Fun\-da\-\c c\~ao para
a Ci\^encia e a Tecnologia (FCT) under the project
 PEst-C/MAT/UI0144/2011.
 A.C. Murza was also supported by the grant
SFRH/ BD/ 64374/ 2009 of FCT.

\end{document}